 \documentclass[12pt, letterpaper]{article}
\usepackage[utf8]{inputenc}
\usepackage[T1]{fontenc}
\usepackage[english]{babel}
\usepackage{amsmath}
\usepackage{amsfonts}
\usepackage{amssymb}
\usepackage{mathtools}
\usepackage{graphicx}
\usepackage{amsthm}
\usepackage{tikz}
\usetikzlibrary{decorations.pathreplacing}
\usepackage{color}
\usepackage[hyperfootnotes=false]{hyperref}
\usepackage{fancyhdr}
\usepackage{subfigure}
\usepackage{xspace}
\usepackage{centernot}
\usepackage{enumitem}

\usepackage{thmtools}
\usepackage{thm-restate}
\usepackage{setspace}

\usepackage{apacite} 
\usepackage{natbib}
\usepackage{lipsum}
\usepackage{adjustbox}
\usepackage[all]{xy}

\usepackage{authblk}

\author[1]{Willem Adriaan Salm}
\affil[1]{Universit\'e Libre de Bruxelles}

\newcommand{\hk}{hyperk\"ahler\xspace}

\title{Elliptic analysis on collapsing gravitational instantons modelled using the Gibbons-Hawking ansatz}

\numberwithin{table}{section}

\declaretheorem[name=Theorem,numberwithin=section]{theorem}
\declaretheorem[name=Lemma,sibling=theorem]{lemma}
\declaretheorem[name=Definition,sibling=theorem]{definition}
\declaretheorem[name=Corollary,sibling=theorem]{corollary}
\declaretheorem[name=Proposition,sibling=theorem]{proposition}

\theoremstyle{remark}

\declaretheorem[name=Remark,sibling=theorem]{remark}

\renewcommand{\d}{\operatorname{d}}
\newcommand{\Z}{\mathbb{Z}}
\newcommand{\del}{\partial}

\newcommand{\N}{\mathbb{N}}
\newcommand{\R}{\mathbb{R}}

\newcommand{\Vol}{\operatorname{Vol}}
\renewcommand{\O}{\mathcal{O}}

\begin{document}

\maketitle
\begin{abstract}
In this paper, we will establish elliptic regularity estimates on families of gravitational instantons whose model metric near infinity collapse to a flat 3 dimensional space. We show that the constants in these estimates can be chosen uniformly for the whole family of metrics. We apply these results to the Laplacian and show it is Fredholm. We also determine between which weighted spaces it is an isomorphism.
\end{abstract}

\tableofcontents

\section{Introduction}
Gravitational instantons are complete, non-compact, \hk manifolds of dimension four with $L^2$-bounded curvature. The first construction of gravitational instantons were found in the late 70's, and in 1989 all asymptotically locally euclidean (ALE) gravitational instantons are classified by \cite{Kronheimer1989}. Over the years other, non-ALE, gravitational instantons were found. \cite{Sun2021} showed that, depending on their asymptotic metric, all gravitational instantons can be separated into 6 classes called, ALE, ALF, ALG, ALG*, ALH and ALH*.
\\

\noindent
\cite{Kronheimer1989} classified all ALE gravitational instantons in terms of the model at infinity and the cohomology classes of the triple of K\"ahler forms. 
These "Torelli theorems" for the other gravitational instantons took time to establish. (See \cite{Minerbe2011} and \cite{Chen2019} for ALF,
\cite{Chen2021} for ALG and ALH, \cite{Chen2021ALG} for ALG*, \cite{Collins2022Torello} for ALH*) 
\\

\noindent
In any case, the moduli spaces for these gravitational instantons are not closed. Therefore, families of gravitational instantons can degenerate. In this paper we focus on one type of degeneration. We will set up the elliptic regularity theory and study the Fredholm properties of the Laplacian in a framework optimized for this degeneration. 
\\

\noindent
To explain this degeneration, recall\footnote{An exposition of all possible asymptotic metrics can be found in \cite{Sun2021}, Chapter 6} that up to some quotient by some finite group action, the metric at infinity of any gravitational instanton can be approximated by the metric found by \cite{Gibbons1978}. This model metric is an explicit \hk metric on a circle bundle over a punctured flat 3-manifold. Even more, the circle radius is bounded for any non-ALE type gravitational instantons. In this paper we focus on families of Riemannian manifolds whose metric approximate an non-ALE type gravitational instantons near infinity whose circle bundle collapses to the flat base space.
\\

\noindent
This analysis is interesting for three reasons. 
First, it will be an essential tool for gluing constructions for these gravitational instantons. Namely, for this method one picks a family of degenerating Riemannian manifolds and a function that measures the failure of the metric to be \hk. To perturb these spaces into a gravitational instantons, one wishes to apply the inverse function theorem on this error map. In order for this to work, one needs that the initial error is sufficiently small and the linearization of this error map is invertible. The first condition can be taken care in the construction of the family of Riemannian manifolds. However, if one does not use the correct Banach spaces, the bound on the inverse of the linearized error map may grow so fast that the inverse function theorem cannot be applied. In this paper we will explain how to set up the analysis such that all our estimates are uniform in our scaling parameter. This can be used to show that the linearized error map has a uniform bounded inverse.
In an upcoming paper the author will use this method in constructing gravitational instantons of type ALG, ALG*, ALH and ALH* by gluing multiple Atiyah-Hitchin and Taub-NUT manifolds to a bulk space that is constructed by the Gibbons-Hawking ansatz.
\\

\noindent
Secondly, it is interesting because it is general enough that it can be applied to other elliptic problems, than just the above gluing construction. Namely, instead of proving each elliptic regularity estimate separately, we will give in Section 		\ref{subsec:analysis-asymptotic:local-estimates} a general recipe where one can convert any standard local estimate on $\R^n$ into an uniform local estimate on these families of approximate gravitational instantons. For example, it can be used in the deformation theory of anti-self-dual instantons over the space of gravitational instantons.
\\

\noindent
Finally, it is interesting because it uses a systematic and uniform approach for all the different types of gravitational instantons, while these different types have totally different behavior. For example, ALE and ALF gravitational instantons are hyperbolic manifolds while the other types are parabolic. Also the model metrics near infinity vary wildly: from a simple quotient of $\R^4$ by a finite subgroup of $SU(2)$ for ALE spaces to ALH* type gravitational instantons whose volume growth is of order $r^{4/3}$. The analysis of each type of gravitational instanton is studied before separately (ALF by \cite{Schroers2020}, ALG by \cite{Chen2020}, ALG* by \cite{Chen2021nxt} and ALH* by \cite{Hein2018}), but in order to show their differences and their similarities, we will do the analysis for all these types simultaneously and only distinguish them when the geometry forces us to.

\subsection{Setup}
In this paper we make the following assumptions and use the following notation.
We assume that we have a family of complete, non-compact, 4-dimensional Riemannian manifolds $(M_\epsilon, g_\epsilon)$, labeled by the parameter $\epsilon \in (0,1)$. 
\\

\noindent
Next, we assume that the model metric near infinity approximates the following version of the Gibbons-Hawking ansatz: Consider, 
\begin{itemize}
	\setlength\itemsep{0em}
	\item $B$, the quotient of $\R^3$ by a non-maximal lattice, where depending on the type of gravitational instanton the topology of $B$ is given in Table \ref{table:conditions-gibbons-Hawking}.
	\item $B'$ as the complement of a large compact set in $B$.
	\item $P_\infty$, a circle bundle over $B'$
	\item $h$, a positive harmonic function on $B'$ that is given in Table \ref{table:conditions-gibbons-Hawking}. We assume $c$ > 0 and $k \in \Z$ is chosen such that $[*\d h] = c_1(P_\infty) \in H^2_{dR}(B')$.
	\item a connection $\eta$ on $P_\infty$ that satisfies the Bogomolny equation $* \d h = \d \eta$.
\end{itemize}
\begin{table}[h!]
	\centering
	\label{table:conditions-gibbons-Hawking}
	\begin{tabular}{c|c|c|c|c|c}
		& ALF & ALG & ALG* & ALH & ALH* \\
		\hline
		$B$ & $\R^3$ & $\R^2 \times S^1$ & $\R^2 \times S^1$ & $\R^+\times T^2$ & $\R^+ \times T^2$ \\
		$h$ & $c + \frac{k}{2 |x|}$ & $c$ & $c + k \: \log |x|$ 
		& $c$ & $c + k \: |x|$ 
	\end{tabular}
	\caption{\textit{Conditions for the Gibbons-Hawking model metric depending on the type of gravitational instanton. For function $h$ we assume that $c > 0$ and $k \in \Z$.}}
\end{table}
\noindent
Given this data, we can equip $P_\infty$ with the following Riemannian metric
\begin{equation}
	\label{eq:gibbons-hawking-metric}
	g^{GH}_\epsilon = (1 + \epsilon h) g_B + \frac{\epsilon^2}{1 + \epsilon h} \eta^2
\end{equation}
We will often use the shorthand $h_\epsilon := 1 + \epsilon h$. Due to the Bogomolny equation, the 
K\"ahler forms $\omega_i^{GH} = \epsilon \d x_i \wedge \eta + h_\epsilon \d x_j \wedge \d x_k $ are an orthonormal closed basis of the self-dual 2-forms and make $g^{GH}_\epsilon$ \hk.
\\

\noindent
Now, let $\Gamma$ be a finite\footnote{When $B = \R^+ \times T^2$, this requires $\Gamma$ to be the trivial group.} group of rotations in $\R^3$, $\R^2$ or $\R^+$ respectively.  Lift $\Gamma$ to a free action on $P_\infty$ and assume that $g^{GH}_\epsilon$ is invariant under this action. Then $P_\infty/\Gamma$ is a \hk manifold and we consider this as our model space at infinity.\\

\noindent
Finally, we wish to compare our model metric $(P_\infty/\Gamma, g_\epsilon)$ to $(M_\epsilon, g_\epsilon)$. So assume we can identify the complement of a large compact set of $M_\epsilon$ with $P_\infty/\Gamma$. To compare the metrics, assume that there is some $\nu > 0$ such that for all $k \in \N$, $|\nabla^k (g_\epsilon-g_\epsilon^{GH})| = \O(r^{-k-\nu})$. We also require that this error can be estimated uniformly in the  parameter $\epsilon$.

\begin{remark}
	For any $x \in B'$, the circle radius of its fiber is of order $\O(\epsilon)$.  Therefore, we can see $\epsilon$ as a scaling parameter and in the limit $\epsilon \to 0$, the space $P_\infty$ collapses to the base space $B'$.
\end{remark}

\begin{remark}
Although the choices for $B$ and $h$ look very restrictive, according to \cite{Sun2021} the metric $g^{GH}_{\epsilon}$ describes a generic model metric for a gravitational instanton. We only restricted our scope in our choice of degeneration.
\end{remark}

\begin{remark}
	\label{remark:radially-invariant-connection}
	Let $r$ be the radial coordinate on $\R^3$, $\R^2$ or $\R^+$ when $B = \R^3$, $\R^2\times S^1$ or $\R^+ \times T^2$ respectively. Considering the definition of $h$ in Table \ref{table:conditions-gibbons-Hawking}, one can show $\d h$ is a multiple of $\Vol_{S^2}$ or $\Vol_{T^2}$ when $B = \R^3$ or $B \not = \R^3$ respectively. Hence we can always find a connection on $P_\infty$ that is $r$-invariant. Even more, the space of connections modulo gauge, $H^1(B,\R)/H^1(B, \Z)$, coincides with the space of $r$-invariant connections and so up to gauge transform we assume that $\eta$ is $r$-invariant.
\end{remark}

\subsection{Results}
\noindent
Our goal in this paper is twofold. First, we want to find an elliptic regularity theory on $(M_\epsilon, g_\epsilon)$ such that all elliptic estimates are uniform in the scaling parameter $\epsilon$. Because these manifolds are non-compact, we have to define suitable weighted norms. For the cases ALG* and ALH* this is non-trivial, as their metric at infinity is neither conical or cylindrical.
We explain in Section \ref{sec:weighted-norms} that after some conformal rescaling and certain universal covers, the manifold has bounded geometry, uniform in the scaling parameter $\epsilon$. This will enable us to define the weighted norms needed for this study.
In Section \ref{subsec:analysis-asymptotic:local-estimates}, we will give a general recipe how to get regularity results for any elliptic operator on these manifolds.
We apply it explicitly for the Laplacian, for which we can conclude
\begin{theorem}
	\label{thm:analysis-asymptotic:schauder-global-holder}
	Let $k \in \N_{\ge 2}$, $\alpha \in (0,1)$ and $\delta \in \R$ and the scaling parameter $\epsilon \in (0,1)$. Consider the weighted H\"older norm $C^{k, \alpha}_\delta$ from Definition \ref{def:asymptotic-geometry:weighted-norm}, the asymptotic regions $P_\infty$ with a  neighborhood $P'_\infty$ and the conformal rescaling $\Omega$ from definition \ref{def:asymptotic-geometry:metric-cf}. There exists a constant $C > 0$, uniformly in $\epsilon$, such that for any $u \in C^{k, \alpha}_{loc}(P'_\infty)$, $\Omega^{-2} \Delta u \in C^{k-2, \alpha}_{\delta}(P'_\infty)$ implies $u \in C^{k,\alpha}(P_\infty)$ and
	$$
	\|u\|_{C^{k, \alpha}_\delta(P_\infty)} \le C \left[
	\| \Omega^{-2} \Delta u \|_{C^{k-2}_\delta (P'_\infty)} + \|u\|_{C^0_\delta(P'_\infty)}
	\right].
	$$
	Furthermore, if $u$ vanishes on $\del P'_\infty$, then
	$$
	\| u\|_{C^{k,\alpha}_{\delta}(P'_\infty)} \le C \left[
	\| \Omega^{-2} \Delta u\|_{C^{k-2,\alpha}_{\delta}(P'_\infty)}
	+ \| u\|_{C^{0}_{\delta}(P'_\infty)}
	\right].
	$$
\end{theorem}
\noindent
Similar regularity results for the Laplacian in terms of Sobolev norms are given in Theorem \ref{thm:analysis-asymptotic:schauder-global-Sobolev} and \ref{thm:analysis-asymptotic:schauder-global-nash-moser}.
\\

\noindent
In the second part of this paper we restrict our attention solely to the Laplacian. In Section \ref{sec:Fredholm-asymptotic} we will study its Fredholm properties. Using the fibration in the Gibbons-Hawking ansatz, we can decompose the relevant function spaces into an $S^1$ invariant and an $S^1$ non-invariant part and we can study the Laplacian on these spaces separately. For the $S^1$-invariant part, the Fredholm properties are already well known, because $h_\epsilon\Delta_{g^{GH}_\epsilon}$ equals $\Delta_{B}$. We will show that the $S^1$ non-invariant part does not introduce new indicial roots and does not change the index of the Laplacian. The key estimate in this section will be

\begin{theorem}
	\label{thm:analysis-asymptotic:special-perp-schauder-polynomial-case}
	On top of the conditions of Theorem \ref{thm:analysis-asymptotic:schauder-global-holder}, assume that $\delta \not \in \Z$.
	When the end is modelled on the ALH gravitational instanton, we also let $\epsilon < \frac{1}{2C} \min_{P_\infty} h_\epsilon$. 
	Then there exists $P_\infty \subseteq P'_\infty$ and a uniform constant $C' > 0$ such that for any $u \in W^{2, 2}_{\delta}(P'_\infty)$ or $u \in C^{2, \alpha}_{\delta}(P'_\infty)$,
	\begin{align*}
		\| u\|_{W^{2,2}_{\delta}(P_\infty)} \le& C' \left[
		\| \Omega^{-2}\Delta u\|_{L^{2}_{\delta}(P'_\infty)}
		+ \| u\|_{L^2_{\delta}(P'_\infty\setminus P_\infty)} 
		\right] \text{ or} \\
		\| u\|_{C^{2,\alpha}_{\delta}(P_\infty)} \le& C' \left[
		\| \Omega^{-2}\Delta u\|_{C^{0, \alpha}_{\delta}(P'_\infty)}
		+ \| u\|_{C^0_{\delta}(P'_\infty\setminus P_\infty)} 
		\right] \text{ respectively.} 
	\end{align*}
	When $u$ vanishes on $\del P'_\infty$,
	\begin{align*}
		\| u\|_{W^{2,2}_{\delta}(P'_\infty)} \le& C' \left[
		\| \Omega^{-2}\Delta u\|_{L^{2}_{\delta}(P'_\infty)}
		+ \| u\|_{L^2_{\delta}(P'_\infty\setminus P_\infty)} 
		\right] \text{ or} \\
		\| u\|_{C^{2,\alpha}_{\delta}(P'_\infty)} \le& C' \left[
		\| \Omega^{-2}\Delta u\|_{C^{0, \alpha}_{\delta}(P'_\infty)}
		+ \| u\|_{C^0_{\delta}(P'_\infty\setminus P_\infty)} 
		\right] \text{ respectively.} 
	\end{align*}
\end{theorem}
\noindent
Using a standard argument one can directly conclude that $\Omega^{-2}\Delta_{g_\epsilon}$ is Fredholm.
\\

\noindent
In Section \ref{subset:analysis-asymptotic:fredholm-theory} we determine the kernel and co-kernel for small weights. First we will focus only on the model operator on $P_\infty$ and determine the (co)-kernel explicitly assuming Dirichlet boundary conditions. We will show that for ALF gravitational instantons, the weights can be chosen such that $\Omega^{-2} \Delta_{g_\epsilon} \colon C^{k+2, \alpha}_\delta(P_\infty) \to C^{k, \alpha}_\delta(P_\infty) $ is an isomorphism. This is not true in general, but in Corollary \ref{cor:analysis-asymptotic:laplace-bijectivity-dirichlet} we show that we can make it isomorphic by adding a single function to its domain.
\\

\noindent
Using this knowledge we finally extend our results globally. 
In Section \ref{sec:Fredholm-global} we will recover the result by \cite{Schroers2020}, that for $\delta \in (-1,0)$ and $ALF$ gravitational instantons, $\Omega^{-2} \Delta_{g_\epsilon} \colon C^{k+2, \alpha}_\delta(M_\epsilon) \to C^{k, \alpha}_\delta(M_\epsilon) $ is an isomorphism. For the other gravitational instantons we prove
\begin{theorem}
	\label{thm:analysis-asymptotic:laplace-bijectivity-global}
	Assume $B \not = \R^3$.
	Let $\rho$ be the radial function from Definition \ref{def:asymptotic-geometry:weighted-operator}.
	Let $\phi$ be a smooth function on $M_\epsilon$, such that $\phi = \rho$ near infinity and vanishes on the interior compact set. Let $\delta \in (-1, 0)$ with $|\delta|$ sufficiently small, $k \in \N_{\ge 2}$ and $\alpha \in (0,1)$. 
	For any $f \in C^{k-2,\alpha}_{\delta}(M_\epsilon)$ or $f \in W^{k,2}_\delta (M_\epsilon)$, there exists a unique $u \in W^{k, 2}_{\delta}(M_\epsilon) \oplus \R \phi$ or $u \in C^{k, \alpha}_{\delta}(M_\epsilon) \oplus \R \phi$ respectively, such that 
	\begin{align*}
			\Omega^{-2} \Delta_{g_\epsilon} u = f.
		\end{align*}
\end{theorem}
\noindent
To show this we use the formal adjoint of $\Omega^{-2} \Delta$ and the maximum principle to get some partial results. Using Corollary \ref{cor:analysis-asymptotic:laplace-bijectivity-dirichlet}, we improve these results and get Theorem \ref{thm:analysis-asymptotic:laplace-bijectivity-global} for the Sobolev case. By embedding the H\"older space inside the Sobolev spaces by tweaking the weight, we show this result is also true for the other case.
\\

\noindent
\textbf{Acknowledgements.} 
This paper contains the analytical results found in my PhD thesis "Construction of gravitational instantons with non-maximal volume growth", funded by the Royal Society research grant RGF\textbackslash R1\textbackslash180086. A special thanks goes to my supervisor Lorenzo Foscolo, without whom this work would not have been possible. 

\section{Weighted elliptic estimates}
\label{sec:weighted-norms}
\noindent
A standard tool for elliptic operators on non-compact manifolds is the use of weighted spaces. We will use weighted operators on an unweighted Banach space instead. 
To explain this, we shortly revisit the Laplacian $\Delta$ on $\R^3$.
Instead of working with $\Delta$, let $\delta \in \R$ and consider the weighted operator $L_\delta := r^{2-\delta} \Delta^{Eucl} (r^\delta \ldots)$ on the fixed Banach space induced by the metric $g_{cf} := r^{-2} g_{Eucl}$ instead.
This operator is strictly elliptic and hence for any pair of small open balls $B_{r_1}(x) \subset \subset B_{r_2}(x)$, we have the Schauder estimate	
\begin{equation}
		\label{eq:analysis-asymptotic:example-weighted-estimate}
	\|u\|_{W^{2,2}_{cf}(B_{r_1}(x))} \le 
	C \left[
	\|r^{2 - \delta} \Delta (r^{\delta}u)\|_{W^{2,2}_{cf}(B_{r_2}(x))} 
	+ \|u\|_{L^{2}(B_{r_2}(x))} 
	\right].
\end{equation}
Because the coefficients of $L_\delta$ can be bounded uniformly in $x$ and $g_{cf}$ is the cylindrical metric on $\R \times S^2$, the constant $C$ can be chosen uniformly in $x$. Hence by a summation argument one can find $0 < R' < R$ such that 
\begin{equation*}
	\|u\|_{W^{2,2}_{cf}(\R^3 \setminus B_R(0))} \le 
	C \left[
	\|r^{2 - \delta} \Delta (r^{\delta}u)\|_{W^{2,2}_{cf}(\R^3 \setminus B_{R'}(0))} 
	+ \|u\|_{L^{2}(\R^3 \setminus B_{R'}(0))} 
	\right].
\end{equation*}
In \cite{Bartnik1986}, we have the weighted norm 
$$\| \ldots \|^2_{W^{k,2}_\delta(\R^3\setminus K)} := \sum_{j=0}^{k} \|r^{j- \delta - \frac{3}{2}}\: \nabla^j \ldots\|^2_{L^{2}(\R^3 \setminus K)}$$
with $\delta \in \R$, $k \in \Z_{\ge 0}$ and $K \subset \R^3$ compact. 
Because (the derivatives of ) $ \d \log r$ are bounded in $g_{cf}$, the metric $\| \ldots \|_{W^{k,2}_\delta}$ is equivalent to $\| r^{- \delta} \ldots \|_{W^{k,2}_{cf}}$. So we can rephrase our estimate in terms of weighted norm:
$$
\|u\|_{W^{2,2}_{\delta}(\R^3 \setminus B_R(0))} \le 
C \left[
\|\Delta u\|_{W^{2,2}_{\delta - 2}(\R^3 \setminus B_{R'}(0))} 
+ \|u\|_{L^{2}_\delta(\R^3 \setminus B_{R'}(0))} 
\right].
$$
The upshot of this calculation is that studying weighted norms is equivalent to studying weighted operators and these weighted operators can be studied using the standard regularity results, which are already well established in literature.

\subsection{Definition of the weighted norms}
\label{subset:analysis-asymptotic:bounded-geometry}
\noindent
In order to apply the above 'weighted operator' method for our gravitational instantons, we need to study the dependence of the constant $C$ on the collapsing parameter $\epsilon$ in equation \ref{eq:analysis-asymptotic:example-weighted-estimate}.
In the above example, we used the translation and rotation invariance of the metric $g_{cf} = \d \log r^2 + g_{S^2}$ and uniform bounds on the coefficients of $L_\delta$ to argue that $C$ does not depend on $x$. Instead using the symmetries of $g_{cf}$, one can make the same conclusion using \textit{bounded geometry}: A complete Riemannian manifold $(M,g)$ has bounded geometry if the injectivity radius is bounded below and the first $k$ derivatives of the Ricci curvature are bounded above, for some $k \in \N_0$.
Riemannian manifolds with bounded geometry are interesting, because they can be equipped with coordinate charts that are suitable for analysis:

\begin{theorem}[Theorem 1.2 in \cite{Hebey1999}]
	\label{thm:analysis-asymptotic:hebey}
	Let $k \in \N$, $\alpha \in (0,1)$ and $Q > 1$. Let $(M,g)$ be a Riemannian manifold with bounded geometry. There exists a constant $r_H > 0$ such that for any $p \in M$,
	there are coordinates $\{x_i\}$ on $B_{r_H}(p)$ that satisfy $Q^{-1} \delta_{\mu\nu} < g_{\mu\nu} < Q \: \delta_{\mu\nu}$ as bilinear forms, and
	\begin{align*}
		\sum_{1 < j < k} r_H^j \: \sup_{ y \in B_r(p)}|\del^{(j)} g(y)| + r_H^{k + \alpha} \sup_{\substack{y,z \in B_r(p)\\ y \not = z}} \frac{|\del^{(k)} g(y) - \del^{(k)} g(z)|}{|y-z|^\alpha} < Q-1.
	\end{align*}
	The constant $r_H$ depends on $k$, $\alpha$, $Q$, the injectivity radius and the $C^0$ norm of (the first $k$ derivatives of) the Ricci tensor.
\end{theorem}
\noindent
The coordinates found in Theorem \ref{thm:analysis-asymptotic:hebey} are called harmonic coordinates, because they satisfy $\Delta x_i = 0$. Like the injectivity radius estimates the largest ball on which the Riemann normal coordinates are defined, measures $r_H$ the largest ball on which the harmonic coordinates has $C^{k, \alpha}$ control of the metric. Therefore, the constant $r_H$ is referred as the harmonic radius in literature.
\\

\noindent
Theorem \ref{thm:analysis-asymptotic:hebey} gives an alternative proof that the constant $C$ in equation \ref{eq:analysis-asymptotic:example-weighted-estimate} does not depend on $x$.
Namely, if the constants $r_1, r_2$ in Equation \ref{eq:analysis-asymptotic:example-weighted-estimate} are less than the harmonic radius, Equation \ref{eq:analysis-asymptotic:example-weighted-estimate} follows directly from the standard Schauder estimate in $\R^3$ and the local $C^{k, \alpha}$ equivalence of the metric. Moreover, one can find harmonic coordinates centred at every $x \in \R^3$.
\\

\noindent
If the bounds on the injectivity radius and the Ricci tensor are uniform in the scaling parameter $\epsilon$, Theorem \ref{thm:analysis-asymptotic:hebey} imply our estimates are uniform in $\epsilon$. With this in mind, we now will define a conformal rescaling for our family of gravitational instantons. 
\begin{definition}
	\label{def:asymptotic-geometry:metric-cf}
	Let $g^{GH} = h_\epsilon g_B +  \epsilon^2 h_\epsilon^{-1} \eta^2$ be the Gibbons-Hawking metric. Let $B$ be defined in Table \ref{table:conditions-gibbons-Hawking} and $r$ be the radial parameter defined in Remark \ref{remark:radially-invariant-connection}.
	Let $\Omega$ be a strictly positive function on $M_\epsilon$ such that
	$$
	\Omega|_{P_\infty} = \begin{cases}
		h_\epsilon^{-\frac{1}{2}} &\text{if } B = \R \times T^2\\
		r^{-1} h_\epsilon^{-\frac{1}{2}} &\text{otherwise.}
	\end{cases}
	$$
	Define the conformally rescaled metric $g_{cf}$ as
	$$
	g_{cf} = \Omega^2 \cdot g_\epsilon.
	$$
\end{definition}
\noindent
The difference between these cases is due to the fact that $S^1$-invariant harmonic functions on $\R \times T^2$ have exponential rather than polynomial growth or decay. 
The conformal rescaling of $h_\epsilon^{-1}$ in $\Omega$ is convenient, because for $S^1$ invariant functions the analysis reduces to the standard analysis on $\R^n$. This is due to the fact that $h_\epsilon \Delta_{g^{GH}_\epsilon} = \Delta^{B}$ for $S^1$ invariant functions. \\

\noindent
To check whether this metric has bounded geometry, one first need to estimate the derivatives of the Ricci tensor. Because the Gibbons-Hawking metric is given explicitly, one can find the Christoffel symbols using the Koszul formula. Using explicit calculation one can show:
\begin{lemma}
	\label{lem:analysis-asymptotic:christoffel-gtilde}
	On the asymptotic region $(P_\infty, g_{cf})$, The Ricci curvature tensor and its first $k$ covariant derivatives are given in terms of the first $k+1$ covariant derivatives of $\d \log \Omega$. In particular, these are uniformly bounded for $\epsilon \in (0,1)$.
\end{lemma}

\noindent
Before we continue our study in the bounded geometry of $P_\infty$, we first define the weighted operator for the Laplacian and check whether it is elliptic.
\begin{definition}
	\label{def:asymptotic-geometry:weighted-operator}
	Let $\Omega$ be as described in Definition \ref{def:asymptotic-geometry:metric-cf} and $r$ be the radial parameter defined in Remark \ref{remark:radially-invariant-connection}.
	Let $\rho$ be a strictly positive function on $M_\epsilon$ such that
	$$
	\rho|_{P_\infty} = \begin{cases}
		r &\text{if } B = \R \times T^2\\
		\log r &\text{otherwise}.
	\end{cases}
	$$
	For all $\delta \in \R$, We define the weighted operator $L_\delta$ as
	$
	e^{- \delta \rho} \Omega^{-2} \Delta_{g_\epsilon} (e^{\delta \rho} \ldots).
	$
\end{definition}
\noindent
Using the Koszul formula, one can show $\d \rho$ has norm one and all its derivatives are bounded uniformly for $\epsilon \in (0,1)$. Therefore, we use this as the radial parameter by which we will measure decay. Using the bounds on $\d \rho$ and its higher derivatives, one can show by explicit calculation that $L_\delta$ is strictly elliptic in the sense of \cite{Gilbarg2001}:
\begin{proposition}
	\label{prop:analysis-asymptotic:laplace-equation-expansion}
	For each $\delta \in \R$, the operator $L_\delta$ is a strictly elliptic operator with bounded coefficients between $C^{k+2, \alpha}_{cf}(P_\infty)$ and $C^{k, \alpha}_{cf}(P_\infty)$, uniformly in $\epsilon \in (0,1)$. That is, if one considers the local coordinates given in Theorem \ref{thm:analysis-asymptotic:hebey} and expands $L_\delta$ as
	$$
	L_\delta = a^{ij} \del_i \del_j + b^{i} \del_i + c,
	$$
	then there exist $\lambda, \Lambda > 0$, independent of $\epsilon$, such that
	$$
	- a^{ij} \xi_i \xi_j \ge \lambda |\xi|^2 \qquad \forall \xi \in \R^4
	$$
	and $\| a^{ij}\|_{C^{0, \alpha}_{cf}}, \| b^{i}\|_{C^{0, \alpha}_{cf}}, \| c\|_{C^{0, \alpha}_{cf}} < \Lambda$.
\end{proposition}

\noindent
We return to the study of the bounded geometry of $P_\infty$. Except for the injectivity radius, all conditions stated in Theorem \ref{thm:analysis-asymptotic:hebey} are satisfied by Lemma \ref{lem:analysis-asymptotic:christoffel-gtilde}. However, for ALG* and ALH* gravitational instantons the injectivity radius decays to zero, because the circle fibers decay at infinity. To remedy this, we replace the fibers with their universal cover. To be precise, we will consider local trivialisations over sufficiently large, contractible open sets and we work on the universal cover over these trivialisations. We claim that, on these local universal covers, the injectivity radius is bounded below. For this we use a result by \cite{Cheeger1982}, which states that it is sufficient to get a lower bound on $\Vol_{cf}(B_1(p))$ for all $p \in P_\infty$. Secondly, we will determine how the Sobolev and H\"older norms change when we project them back to neighborhoods on $P_\infty$.

\begin{lemma}
	\label{lem:analysis-asymptotic:injectivity-radius}
	On local universal covers of $P_\infty$, the injectivity radius is bounded below, uniformly in $\epsilon \in (0,1)$.
\end{lemma}
\begin{proof}
	We explain the case $B = \R \times T^2$.
	Pick $p = (x_0,0) \in P_\infty$ and choose $\varrho > 0$ such that the ball $B_\varrho(x_0) \subset B$ is contractible. Next, we trivialize $P|B_\varrho(x_0) \simeq B_\varrho(x_0) \times S^1$ and consider the following rectangular neighborhood on its universal cover:
	$$
	R_\varrho(p) := \left\{(x, t) \in B_\varrho(x_0) \times \R \colon |t| < \frac{h_\epsilon(x)}{\epsilon} \varrho \right\}
	$$
	We claim $R_\varrho(p)$ lies inside a circumscribed ball of fixed length. To show this, pick $(x, t) \in R_\rho(p)$ and consider the path that goes parallel along the coordinate axis. Using the gauge fix in Remark \ref{remark:radially-invariant-connection}, the length of this path is bounded above by some uniform constant $C>0$, and so $R_\varrho(p)$ lies inside the ball of radius $C$ centred at $p$.
	The volume of $R_\varrho(p)$ is equal to
	$$
	\Vol_{cf}(R_\varrho(p)) = \int_{x \in B_\varrho(x_0)} 
	\int_{- \frac{h_\epsilon}{\epsilon} \varrho}^{ \frac{h_\epsilon}{\epsilon} \varrho} 
	\frac{\epsilon}{h_\epsilon} \Vol_{g_B} \wedge \d t = 2\Vol_{g_B}(B_\rho(x_0)) = \frac{8}{3} \pi \varrho^3.
	$$
	According to \cite{Cheeger1982}, the injectivity radius at $p$ on $R_\rho(p)$ is bounded below uniformly in $\epsilon \in(0,1)$.
	\\
	
	\noindent
	When $B = \R^2 \times S^1$ the injectivity radius will still decay to zero at infinity. This is due to the term $\frac{1}{e^\rho} g_{S^1}$ in the metric. However, when we consider $P_\infty$ as a $T^2$ bundle and use this unwrapping trick for both decaying fibers at the same time, we still get a lower bound on the injectivity radius.
\end{proof}

\noindent
Finally, we relate periodic functions on these universal local covers to functions on $P_\infty$. When using Sobolev norms, one has to take account of the number of fundamental domains are covered inside a ball of certain radius. This is not necessary for H\"older norms, due to the use of supremum norms. 

	\begin{lemma}
		\label{lem:analysis-asymptotic:function-covering-space}
		Let $V \subseteq P_\infty$ be open such that $V$ restricts to a trivial $S^1$-bundle or $T^2$-bundle respectively over a contractible base space.
		Let $\hat{V}$ be the universal cover of $V$. 
		Then, for any $u \in C^{k, \alpha}_{cf}(V)$,
		$$
		\| u \|_{C^{k, \alpha}_{cf}(V)} =\|\hat{ u} \|_{C^{k, \alpha}_{cf}(\hat{V})},
		$$
		where $\hat{u}$ is the lift of $u$ in $C^{k, \alpha}_{cf}(\hat{V})$.
	\end{lemma}

		\begin{lemma}
			\label{lem:analysis-asymptotic:compare-metrics-universal-bundle}
			Let $r>0$ be less than the injectivity radius found in Lemma \ref{lem:analysis-asymptotic:injectivity-radius}. Let $p \in P_\infty$, let $B_r(p)$ be the ball of radius $r$ in $P_\infty$ and let $\hat B_r(p)$ be the ball of radius $r$ on the local universal cover of $P_\infty$ at $p$.
			Consider the function
			$$
			v^2 = \begin{cases}
				\frac{e^{\rho} h_\epsilon}{\epsilon} &\text{if } B = \R^3 \\
				\frac{e^{2 \rho} h_\epsilon}{\epsilon} &\text{if } B = \R^2 \times S^1 \\
				\frac{h_\epsilon}{\epsilon} &\text{otherwise}.
			\end{cases}
			$$
			Then, there exist $1 < M_1 < M_2$ and $0  <C_1 < C_2$, independent of $p$ and $\epsilon$, such that for all $u \in L^2(B_r(p))$
			$$
			C_1 \|v \cdot u \|_{L^2_{cf}(B_{r/M_2}(p))} 
			\le \|\hat u \|_{L^2_{cf}(\hat{B}_{r/M_1}(p))} 
			\le C_2 \|v \cdot u \|_{L^2_{cf}(B_r(p))},
			$$
			where $\hat u$ is the periodic lift of $u$ on $\hat{B}_r(x)$.
		\end{lemma}
		\noindent
		To prove Lemma \ref{lem:analysis-asymptotic:compare-metrics-universal-bundle}, one has to count the number of fundamental domains inside and around $\hat{B}_{r/M_1}(p)$ and compare this to the function $v$. This gives the constants $C_1$ and $C_2$. To do this explicitly it is easier to use the rectangular neighborhoods used in the proof of Lemma \ref{lem:analysis-asymptotic:injectivity-radius}. For this, one needs to estimate its inscribed and circumscribed balls, which gives the constants $M_1$ and $M_2$.
		\\
		
		\noindent
		In summary, the following weighted norms will be the most suited norms for analysis on $(M_\epsilon, g_\epsilon)$:
		\begin{definition}
			\label{def:asymptotic-geometry:weighted-norm}
			Let $\Omega$ and $g_{cf}$ be as described in Definition \ref{def:asymptotic-geometry:metric-cf} and let $\rho$ be as in Definition \ref{def:asymptotic-geometry:weighted-operator}.
			For any $k \in \N$, $\alpha \in (0,1)$, $\delta \in \R$, we define the weighted H\"older norm on $U \subseteq P_\infty$ as
			$$
			\|u\|_{C^{k, \alpha}_{\delta}(U)} = \|e^{-\delta \rho}\cdot u \|_{C^{k,\alpha}_{cf}(U)}.
			$$
			For any $k \in \N$, $\delta \in \R$, we define the weighted $L^2$ and Sobolev norm on $U \subseteq P_\infty$ as
			\begin{align*}
				\langle u, v \rangle_{L^{2}_{\delta}(U)} =& \langle e^{- \delta \rho} \: u, e^{- \delta \rho}\: v \rangle_{\tilde L^2(U)} \\	
				\|u\|^2_{W^{k, 2}_{\delta}(U)} =& \sum_{n=0}^k \| \: |\nabla^n (e^{- \delta \rho} \cdot u)|_{cf}  \|^2_{ \tilde L^2(U)}
			\end{align*}
			where $\tilde L^2(U)$ is the $L^2$ norm with respect to the volume form $\widetilde{\Vol}$ that is induced from Lemma \ref{lem:analysis-asymptotic:compare-metrics-universal-bundle}. When $g_\epsilon = g^{GH}_\epsilon$, 
			$$
			\widetilde{\Vol}|_{P_\infty} =
			\begin{cases}
				\d \rho \wedge \Vol_{S^2} \wedge \eta &\text{if } B = \R^3 \\
				\d \rho \wedge \Vol_{S^1 \times S^1} \wedge \eta &\text{if } B = \R^2 \times S^1\\
				\d \rho \wedge \Vol_{T^2} \wedge \eta &\text{if } B = \R \times T^2.
			\end{cases}
			$$
		\end{definition}
		\subsection{Recipe for weighted elliptic estimates}
		\label{subsec:analysis-asymptotic:local-estimates}
		\noindent
		With all these ingredients we now have a method to establish local elliptic estimates on $P_\infty$. We will work out one example in detail: To rephrase the estimate\footnote{This regularity result is a combination of \cite{Gilbarg2001}, Problem 6.1, Theorem 9.19 and \cite{Folland1995} Theorem 6.33.}
		$$
		\|u\|_{C^{k, \alpha}_{\R^n}} \le C \left[
		\|\Delta u\|_{C^{k-2, \alpha}_{\R^n}}
		+ \|u\|_{C^{0}}
		\right],
		$$ 
		we follow the following steps:
		\\
		
		\noindent
		\textit{Step 1:} First we define the domains on which we apply the estimates. According to Theorem \ref{thm:analysis-asymptotic:hebey} there exists $r_H > 0$ independent of $x \in P_\infty$ and $ \epsilon > 0$, such that the ball $\hat{B}_{r_H}(x)$ inside the local universal cover can be equipped with harmonic coordinates.  We fix $0 < r < r' < r_H$, $ k \in \N_{\ge 2}$, $ \alpha \in (0,1)$ and $\delta \in \R$. For our local elliptic estimates we only consider function $u \in C^{k,\alpha}_\delta (B_{r'}(x))$.		
		\\
			
		\noindent
		\textit{Step 2:} We lift $u \in C^{k,\alpha}_\delta (B_{r'}(x))$ to a periodic function  $\hat u$ on the local universal cover inside $\hat B_{r'}(x)$. Combining Definition \ref{def:asymptotic-geometry:weighted-norm} with Lemma \ref{lem:analysis-asymptotic:function-covering-space}, we relate 
		 $		\|u\|_{C^{k, \alpha}_{\delta}(B_r'(x))} = \|e^{- \delta \rho }\hat{u}\|_{C^{k, \alpha}_{cf}(\hat{B}_r'(x))}$.
		Because the local universal cover has bounded geometry, the $C^{k, \alpha}_{cf}(\hat{B}_r'(x))$ norm is equivalent to the standard $\R^n$ norm induced by the harmonic coordinates. Using these coordinates, we get the elliptic estimate
		\begin{align*}
			\|e^{- \delta \rho }\hat u\|_{C^{k, \alpha}_{cf}(\hat{B}_r(x))} 
			\le& C \left[
			\|L_\delta(e^{- \delta \rho }\hat u)\|_{C^{k-2, \alpha}_{cf}(\hat{B}_r'(x))}
			+ \|e^{- \delta \rho }\hat u\|_{C^{0}(\hat{B}_r'(x))}
			\right].
		\end{align*}
		\noindent
		\textit{Step 3:} Using the fact that $L_\delta$ is invariant under deck transformations, we project down to balls on $P_\infty$, and by Lemma \ref{lem:analysis-asymptotic:function-covering-space}:
		\begin{align*}
			\|e^{- \delta \rho } u\|_{C^{k, \alpha}_{cf}({B}_r(x))} 
			\le& C \left[
			\|L_\delta(e^{- \delta \rho } u)\|_{C^{k-2, \alpha}_{cf}({B}_r'(x))}
			+ \|e^{- \delta \rho } u\|_{C^{0}({B}_r'(x))}
			\right].
		\end{align*}
		Using the definition of $L_\delta$,
		\begin{align*}
			\|e^{- \delta \rho } u\|_{C^{k, \alpha}_{cf}({B}_r(x))} 
			\le& C \left[
			\|e^{- \delta \rho } \Omega^{-2}\Delta_{g_\epsilon} u\|_{C^{k-2, \alpha}_{cf}({B}_r'(x))}
			+ \|e^{- \delta \rho }\hat u\|_{C^{0}({B}_r'(x))}
			\right],
		\end{align*}
		and we conclude:
		\begin{proposition}
			\label{prop:analysis-asymptotic:schauder-local-holder}
			Let $k \in \N_{\ge 2}$, $\delta \in \R$ and $\alpha \in (0,1)$.
			For sufficiently small $0 < r < r'$, there exists an uniform constant $C > 0$ such that for all $x \in P_\infty$ and any distribution
			$u$ on $B_{r'}(x)$
			with $\Omega^{-2}\Delta_{g_\epsilon} u \in C^{k-2, \alpha}_{\delta}\left(B_{r'}(x)\right)$,
			$$u \in {C}^{k,\alpha}_{\delta}(B_{r'}(x))$$
			and
			$$
			\| u\|_{C^{k,\alpha}_{\delta}\left(B_{r}(x)\right)} \le C \left[
			\|\Omega^{-2} \Delta_{g_\epsilon} u\|_{C^{k-2,\alpha}_{\delta}\left(B_{r'}(x)\right)}
			+ \| u\|_{C^{0}_{\delta}\left(B_{r'}(x)\right)}
			\right].
			$$
		\end{proposition}
		\noindent
		Similarly, we get a local Schauder estimate using Sobolev norms. For this we use the results on $\R^n$ from \cite{Evans1998} (Theorem 1 in section 6.3.1) and \cite{Bandle1998} (Theorem 7-12).
		
		\begin{proposition}
			\label{prop:analysis-asymptotic:schauder-local-Sobolev}
			Let $k \in \N_{\ge 2}$ and $\delta \in \R$.
			For sufficiently small $0 < r < r'$, there exists an uniform constant $C > 0$ such that for all $x \in P_\infty$ and any distribution
			$u$ on $B_{r'}(x)$
			with $\Omega^{-2}\Delta_{g_\epsilon} u \in C^{k-2, \alpha}_{\delta}\left(B_{r'}(x)\right)$,
			$$u \in {W}^{k,2}_{\delta}(B_{r'}(x))$$
			and
			$$
			\| u\|^2_{W^{k,2}_{\delta}\left(B_{r}(x)\right)} \le C \left[
			\|\Omega^{-2} \Delta_{g_\epsilon} u\|^2_{W^{k-2,2}_{\delta}\left(B_{r'}(x)\right)}
			+ \| u\|^2_{L^{2}_{\delta}\left(B_{r'}(x)\right)}
			\right].
			$$
		\end{proposition}
		\begin{proposition}
			\label{prop:analysis-asymptotic:schauder-local-nash-moser}
			Let $\delta \in \R$ and $\alpha \in (0,1)$.
			For sufficiently small $0 < r < r'$, there exists an uniform constant $C > 0$ such that for all $x \in P_\infty$ and any
			$u \in C^{2, \alpha}_{\delta}(B_{r'}(x)),$
			$$
			\| u\|_{C^{0,\alpha}_{\delta}\left(B_{r}(x)\right)} \le C \left[
			\|\Omega^{-2} \Delta_{g_\epsilon} u\|_{C^{0,\alpha}_{\delta}\left(B_{r'}(x)\right)}
			+ \| u\|_{L^2_{\delta}\left(B_{r'}(x)\right)}
			\right].
			$$
		\end{proposition}

		\noindent
		We are now ready to prove Theorem \ref{thm:analysis-asymptotic:schauder-global-holder}. For this theorem we need a neighborhood $P'_\infty$ of $P_\infty$. Recall that topologically $P_\infty$ is a circle bundle over $ [R_0, \infty) \times \Sigma$, where $R_0 > 0$ and $\Sigma$ is a compact space. A suitable choice of $P'$ can be made by picking $R_1$ slightly smaller than $R_0$ and define $P'_\infty$ as the circle bundle over $ [R_1, \infty) \times \Sigma$. 
		\begin{proof}[Proof of Theorem \ref{thm:analysis-asymptotic:schauder-global-holder}]
			Let $r$ and $r'$ be as described in the steps for Proposition \ref{prop:analysis-asymptotic:schauder-local-holder}. Because $u \in C^{k, \alpha}_{loc}(P'_\infty)$, $u$ must lie in $C^{k, \alpha}_{\delta}(B_{r'}(x))$ for all $x \in P_\infty$. Proposition \ref{prop:analysis-asymptotic:schauder-local-holder} states that
			\begin{align*}
				\| u\|_{C^{k,\alpha}_{\delta}\left(B_{r}(x)\right)} 
				\le& C \left[
				\| \Omega^{-2} \Delta_{g_\epsilon} u\|_{C^{k-2,\alpha}_{\delta}\left(B_{r'}(x)\right)}
				+ \| u\|_{C^{0}_{\delta}\left(B_{r'}(x)\right)}
				\right] \\
				\le& C \left[
				\|\Omega^{-2} \Delta_{g_\epsilon} u\|_{C^{k-2,\alpha}_{\delta}\left(P'_\infty\right)}
				+ \| u\|_{C^{0}_{\delta}\left(P'_\infty\right)}
				\right].
			\end{align*}
			Varying $x \in P_\infty$, we conclude
			\begin{align*}
				\| u\|_{C^{k,\alpha}_{\delta}\left(P_\infty\right)}
				= 
				\sup_{x \in P_\infty}
				\| u\|_{C^{k,\alpha}_{\delta }\left(B_{r}(x)\right)} 
				\le C \left[
				\|\Omega^{-2} \Delta_{g_\epsilon} u\|_{C^{k-2,\alpha}_{\delta}\left(P'_\infty\right)}
				+ \| u\|_{C^{0}_{\delta}\left(P'_\infty\right)}
				\right].
			\end{align*}
			For the boundary regularity estimate we use the same method, combined with Corollary 6.7 from \cite{Gilbarg2001}, which states that, for any $x$ close to the boundary,
			\begin{align*}
				\| u\|_{C^{k,\alpha}_{\delta}\left(B_{r}(x) \cap P'_\infty\right)} 
				\le& C \left[
				\|\Omega^{-2} \Delta_{g_\epsilon} u\|_{C^{k-2,\alpha}_{\delta}\left(P'_\infty\right)}
				+ \| u\|_{C^{0}_{\delta}\left(P'_\infty\right)}
				\right].
			\end{align*}
		\end{proof}
		\noindent
		The same method can be used to extend Proposition \ref{prop:analysis-asymptotic:schauder-local-nash-moser} to a global version. For Proposition \ref{prop:analysis-asymptotic:schauder-local-Sobolev}, we need to use a summation method, similarly to Proposition 6.1.1 in \cite{Pacard2008}. Namely, we pick $\kappa > 0$ and write $P'_\infty$ as the union of annuli $A_n := \pi^{-1}(B_{R_0+\kappa(n+1)} \setminus B_{R_0 + \kappa n})$, and we sum the estimates for all annuli. Because the radius of the circle fiber is uniformly bounded above we can cover $A_n$ with a fixed number of balls such that on each ball we can apply Proposition \ref{prop:analysis-asymptotic:schauder-local-Sobolev}. For large enough $\kappa$, we get the estimate
		\begin{align*}
			\| u\|^2_{W^{k,2}_{\delta}\left(A_n\right)} 
			\le& C \sum_{m = n-1}^{n+1}\left[
			\|\Omega^{-2} \Delta_{g_\epsilon} u\|^2_{W^{k-2,2}_{\delta}\left(A_{m}\right)}
			+ \| u\|^2_{L^{2}_{\delta}\left(A_{m}\right)}\right]
		\end{align*}
		for all $n \in \Z_{\ge 0}$.
		Taking the union over the first $N$ annuli yields
		\begin{align*}
			\| u\|_{W^{k,2}_{\delta}\left(\bigcup_{n=1}^N A_n\right)} 
			\le& 3C \sum_{n = 0}^{N+1} \left[
			\|\Omega^{-2} \Delta_{g_\epsilon} u\|_{W^{k-2,2}_{\delta}\left(A_{n}\right)}
			+ \| u\|_{L^{2}_{\delta}\left(A_{n}\right)}\right] \\
			\le& 3C \left[
			\|\Omega^{-2} \Delta_{g_\epsilon} u\|_{W^{k-2,2}_{\delta}\left(\bigcup_{n=0}^{N+1}A_{n}\right)}
			+ \| u\|_{L^{2}_{\delta}\left(\bigcup_{n=0}^{N+1}A_{n}\right)}\right].
		\end{align*}
		If one assumes that $u$ vanishes on the boundary of $P'_\infty$,
		\begin{align*}
			\| u\|_{W^{k,2}_{\delta}\left(\bigcup_{n=0}^N A_n\right)} 
			\le& 3C \left[
			\|\Omega^{-2} \Delta_{g_\epsilon} u\|_{W^{k-2,2}_{\delta}\left(\bigcup_{n=0}^{N+1}A_{n}\right)}
			+ \| u\|_{L^{2}_{\delta}\left(\bigcup_{n=0}^{N+1}A_{n}\right)}\right]. 
		\end{align*}
		Taking the limit $N \to \infty$ we conclude:
		\begin{theorem}
			\label{thm:analysis-asymptotic:schauder-global-Sobolev}
			Let $k \in \N_{\ge 2}$ and $\delta \in \R$.
			There exists a uniform constant $C > 0$ such that for any $L^2_\delta$-bounded $u \in W^{k,2}_{loc}(P'_\infty)$, $ \Omega^{-2}\Delta_{g_\epsilon} u \in W^{k-2,2}_{\delta}(P'_\infty)$ imply
			$u \in W^{k,2}_{\delta}(P_\infty)$
			and 
			$$
			\| u\|_{W^{k,2}_{\delta}(P_\infty)} \le C \left[
			\| \Omega^{-2} \Delta_{g_\epsilon} u\|_{W^{k-2,2}_{\delta}(P'_\infty)}
			+ \| u\|_{L^{2}_{\delta}(P'_\infty)}
			\right].
			$$
			Furthermore, if $u$ vanishes on $\del P'_\infty$, then
			$$
			\| u\|_{W^{k,2}_{\delta}(P'_\infty)} \le C \left[
			\| \Omega^{-2} \Delta_{g_\epsilon} u\|_{W^{k-2,2}_{\delta}(P'_\infty)}
			+ \| u\|_{L^{2}_{\delta}(P'_\infty)}
			\right].
			$$
		\end{theorem}
				\begin{theorem}
			\label{thm:analysis-asymptotic:schauder-global-nash-moser}
			Let $k \in \N_{\ge 2}$, $\alpha\in (0,1)$ and $\delta \in \R$.
			There exists a uniform constant $C > 0$ such that for any $L^2_\delta$-bounded $u \in L^{2}(P'_\infty)$, $ \Omega^{-2}\Delta_{g_\epsilon} u \in C^{0,\alpha}_{\delta}(P'_\infty)$ imply
			$
			u \in C^{0, \alpha}_{\delta}(P_\infty)
			$
			and
			$$
			\| u\|_{C^{0,\alpha}_{\delta}(P_\infty)} \le C \left[
			\| \Omega^{-2} \Delta_{g_\epsilon} u\|_{C^{0,\alpha}_{\delta}(P'_\infty)}
			+ \| u\|_{L^{2}_{\delta}(P'_\infty)}
			\right].
			$$
		\end{theorem}

	\section{Fredholm theory for the Laplacian}
	\label{sec:Fredholm-asymptotic}
	From now on we will focus on the Laplacian. Given these regularity estimates, the next step is to show $\Omega^{-2} \Delta_{g_\epsilon}$ is Fredholm. A standard argument (e.g. See \cite{Pacard2008}) shows that it is sufficient to prove Theorem \ref{thm:analysis-asymptotic:special-perp-schauder-polynomial-case}, because the term $P'_\infty \setminus P_\infty$ in this theorem can be chosen compact. The prove of this theorem will heavily base on the observation that $h_\epsilon \Delta_{g^{GH}_\epsilon} = \Delta^{B}$ for functions that are pulled back from $B$. On the base spaces, the Fredholm estimates are well known, and hence we only need to study functions on the fiber. We use the following Fourier decomposition:
		\begin{definition}
			\label{def:analysis-asymptotic:S1-invarariant-functions}
			For any continuous function $u$ on $P'_\infty$ define the $S^1$ invariant part of $u$ as
			$$
			u_{b}(x,t) = \frac{1}{2 \pi }\int_{\pi^{-1}(x)} u \: \eta
			$$
			and the $S^1$ non-invariant part
			$u_{f} = u - u_{b}.$ The operators that map $u$ to $u_b$ and $u_f$ will be denoted as $\pi_b$ and $\pi_f$ respectively. 
		\end{definition}
		\noindent
		By construction, the space of continuous functions on $P'_\infty$ has a direct sum decomposition into $S^1$ invariant and $S^1$ non-invariant functions. 
		Related to this splitting there are three analytical properties which turn out to be useful.
		\begin{lemma}
			\label{lem:analysis-asymptotic:laplacian-compatible-splitting}
			The operators $\Delta_{g^{GH}_\epsilon}$ and $\pi_b$ commute.
		\end{lemma}
		\begin{proof}
			Take a local trivialization of the circle bundle and perform this calculation explicitly using the $S^1$-invariance of the metric.
		\end{proof}
		\begin{lemma}
			\label{lem:analysis-asymptotic:boundedness-av-operator}
			On any $S^1$ invariant domain $U$, the operators 
			\begin{align*}
				\pi_b \colon C^0(U) \to C^0(U),
				\pi_b \colon C^{0,\alpha}_{cf}(U) \to C^{0,\alpha}_{cf}(U), \text{ and }
				\pi_b \colon \tilde L^{2}(U) \to \tilde L^2(U)
			\end{align*}
			are bounded. The same holds for $\pi_f$.
		\end{lemma}
		\begin{proof}
			The only non-trivial thing to show is that the H\"older semi-norm is bounded. So let $u \in C^{0, \alpha}_{cf}(U)$ and $x,y \in U$ such that $\left|\frac{u(x) - u(y)}{d(x,y)^\alpha}\right| \le \|u\|_{C^{0, \alpha}_{cf}(U)}$. By rotating along the fiber $\left|\frac{u(e^{it} \cdot x) - u(e^{it} \cdot y)}{d(e^{it} \cdot x,e^{it} \cdot y)^\alpha}\right| \le \|u\|_{C^{0, \alpha}_{cf}(U)}$. Because the metric is $S^1$-invariant, $\left|\frac{u(e^{it} \cdot x) - u(e^{it} \cdot y)}{d(x, y)^\alpha}\right| \le \|u\|_{C^{0, \alpha}_{cf}(U)}$ and so averaging over $t \in [0,2 \pi]$ yields
			\begin{align*}
				\left|\frac{1}{2 \pi}\int_{0}^{2 \pi}
				\frac{u(e^{i t} \cdot x) - u(e^{i t} \cdot y)}{d( x, y)^{\alpha}}
				\d t \right|
				= 
				\frac{|u_b(x) - u_b(y)|}{d( x, y)^{\alpha}}
				\le 
				\|u\|_{C^{0, \alpha}_{cf}(U)}.
			\end{align*}
		\end{proof}
\begin{proposition}[Poincar\'e inequality]
			\label{prop:analysis-asymptotic:poincare-inequality}
			Let $x \in P'_\infty$ and denote the orbit of $x$ as $S^1 \cdot \{x\}$.
			For any continuous function $u$ that satisfies $u = u_f$,
			\begin{align*}
				\|u\|_{C^{0}_{cf}(S^1 \cdot \{x\})} 
				\le& 2\pi \frac{\epsilon \: \Omega}{\sqrt{h_\epsilon}} \| \d u \|_{C^0_{cf}({S^1 \cdot \{x\})}} \\
				\|u\|_{L^2(S^1 \cdot \{x\})} 
				\le& \frac{\epsilon \: \Omega}{\sqrt{h_\epsilon}} \| \d u \|_{\tilde L^{2}(S^1 \cdot \{x\})}.
			\end{align*}
		\end{proposition}
		\begin{proof}
			Let $(x,t) \in P'_\infty$. Because $u$ is $S^1$ non-invariant, there exists a $t_0 \in S^1$ such that 
			$u(x, t_0) = 0$. By the fundamental theorem of calculus, $u(x,t) = \int_{t_0}^t \frac{\del u}{\del t}  \d t$.
			From definition \ref{def:asymptotic-geometry:metric-cf} we estimate the circle radius, and hence
			\begin{align*}
				u(x,t) 
				\le  \int_{t_0}^t \| \d u \|_{cf} \cdot \| \del_t \|_{cf}  \d t 
				\le 2\pi \frac{\epsilon \: \Omega}{\sqrt{h_\epsilon}} \| \d u \|_{C^0_{cf}({S^1 \cdot \{x\})}}.
			\end{align*}
			In order to find the $L^2$ estimate, write $u(x,t) = \sum_n u_n(x) e^{int}$ and note that
			\begin{align*}
				\|u\|^2_{\tilde L^2(S^1 \cdot \{x\})} =\int |u|^2 \d t
				=  \sum_{n=1}^\infty   u_n^2 
				\le  \sum_{n=0}^\infty  n^2 u_n^2 
				\le \|\d u (\del_t)\|^2_{\tilde L^2(S^1 \cdot \{x\})}.
			\end{align*}
			Therefore, 
			$$
			\|u\|^2_{\tilde L^2(S^1 \cdot \{x\})}  \le \|\d u\|^2_{\tilde L^2(S^1 \cdot \{x\})} \cdot \|\del_t\|^2_{C^0(S^1 \cdot \{x\})}.
			$$
			
		\end{proof}
\noindent
If the circle fiber collapses at infinity, this Poincar\'e inequality will enable us to improve Theorem \ref{thm:analysis-asymptotic:schauder-global-holder} into Theorem \ref{thm:analysis-asymptotic:special-perp-schauder-polynomial-case}. This is true for all cases except when $P_\infty$ is a trivial circle bundle, i.e. when the end is modelled on the ALH gravitational instanton. For this case we need the extra requirement that the collapsing parameter $\epsilon$ is sufficiently small.
\begin{proof}[Proof of Theorem \ref{thm:analysis-asymptotic:special-perp-schauder-polynomial-case}] 
			\noindent
			Assume without loss of generality that $B = \R^3$ and $u \in W^{2,2}_\delta(P'_\infty)$. Consider the case $u = u_f$ and $g_\epsilon = g^{GH}_\epsilon$.
			The elliptic regularity estimate from Theorem \ref{thm:analysis-asymptotic:schauder-global-Sobolev} states,
			\begin{align*}
				\| u_f\|_{W^{2,2}_{\delta}(P_\infty)} 
				\le& C \left[
				\| \Omega^{-2} \Delta_{g^{GH}_\epsilon} u_f\|_{L^{2}_{\delta}(P'_\infty)}
				+ \| u_f\|_{L^{2}_{\delta}(P'_\infty)}
				\right] \\
				\le& C \left[
				\| \Omega^{-2} \Delta_{g^{GH}_\epsilon} u_f\|_{L^{2}_{\delta}(P'_\infty)}
				+ \| u_f\|_{L^{2}_{\delta}(P'_\infty \setminus P_\infty)}
				+ \| u_f\|_{L^{2}_{\delta}(P_\infty)}
				\right].
			\end{align*}
			Using the Poincaré inequality, we rewrite this as
			\begin{align*}
				\| u_f\|_{W^{2,2}_{\delta}(P_\infty)} 
				\le& C \left[
				\| \Omega^{-2} \Delta_{g^{GH}_\epsilon} u_f\|_{L^{2}_{\delta}(P'_\infty)}
				+ \| u_f\|_{L^{2}_{\delta}(P'_\infty \setminus P_\infty)} + \right. \\
				& \qquad\qquad\qquad\qquad\left. + \left(\max_{P_\infty} \frac{\epsilon \Omega}{\sqrt{h_\epsilon}}\right) \cdot\| \d u_f\|_{L^{2}_{\delta}(P_\infty)}
				\right].
			\end{align*}
			The term $\frac{\Omega}{h_\epsilon}$ is at least of order $e^{-\rho}$ when the model end is ALF, ALG, ALG* or ALH* and so in those cases we can pick $P_\infty$ such that $\max_{P_\infty} \frac{\epsilon \Omega}{\sqrt{h_\epsilon}} < \frac{1}{2}$. In the ALH case, this condition is part of the theorems assumptions. Hence, for all model ends we conclude
			\begin{align*}
				\frac{1}{2}\| u_f\|_{W^{2,2}_{\delta}(P_\infty)} 
				\le& C \left[
				\| \Omega^{-2}  \Delta_{g^{GH}_\epsilon} u_f\|_{L^{2}_{\delta}(P'_\infty)}
				+ \| u_f\|_{L^{2}_{\delta}(P'_\infty \setminus P_\infty)}
				\right].
			\end{align*}			
			\noindent
			Secondly, consider the case $u = u_b$ and $g_\epsilon = g^{GH}_\epsilon$. For $S^1$ invariant functions, $\Omega^{-2} \Delta_{g^{GH}_\epsilon}$ reduces to the standard Laplacian $\Delta^B$ on the base space. The operator $\Delta^B$ is Fredholm in the norms given by \cite{Bartnik1986} when $\delta \not \in \Z$. By Lemma 	\ref{lem:analysis-asymptotic:christoffel-gtilde}, these norms are equivalent to the Sobolev norms introduced in Definition \ref{def:asymptotic-geometry:weighted-norm}. Therefore, there exists a uniform constant $C > 0$ independent of $u$ such that 
			\begin{align*}
				\| u_b\|_{W^{2,2}_{\delta}(P_\infty)} 
				\le& C \left[
				\| \Omega^{-2} \Delta u_b\|_{L^{2}_{\delta}(P'_\infty)}
				+ \| u_b\|_{L^{2}_{\delta}(P'_\infty \setminus P_\infty)}
				\right].
			\end{align*}
			Combining these estimates and using Lemma \ref{lem:analysis-asymptotic:laplacian-compatible-splitting} and \ref{lem:analysis-asymptotic:boundedness-av-operator}, we conclude that for any $u = u_b + u_f \in W^{2,2}_\delta (P'_\infty)$,
			\begin{align*}
				\| u\|_{W^{2,2}_{\delta}(P_\infty)} 
				\le& \| u_b\|_{W^{2,2}_{\delta}(P_\infty)}  + \| u_f\|_{W^{2,2}_{\delta}(P_\infty)}  \\
				\le& 4C \left[
				\| \Omega^{-2} \Delta_{g^{GH}_\epsilon} u\|_{L^{2}_{\delta}(P'_\infty)}
				+ \| u\|_{L^{2}_{\delta}(P'_\infty \setminus P_\infty)}
				\right],
			\end{align*}
			Finally, we consider the general case where $\| \nabla^k (g_\epsilon - g^{GH}_\epsilon) \| = \O(r^{-k - \upsilon})$. Using elliptic regularity, one can find some neighborhood $P''_\infty \supset P'_\infty$ and some constant $C ' > 0$ such that
			\begin{align*}
				\| u\|_{W^{2,2}_{\delta}(P_\infty)} 
				\le& 4C \left[
				\| \Omega^{-2} \Delta_{g^{GH}_\epsilon} u\|_{L^{2}_{\delta}(P'_\infty)}
				+ \| u\|_{L^{2}_{\delta}(P'_\infty \setminus P_\infty)}
				\right] \\
				\le& C' \left[
				\| \Omega^{-2} \Delta_{g_\epsilon} u\|_{L^{2}_{\delta}(P''_\infty)}
				+ \| u\|_{L^{2}_{\delta}(P'_\infty \setminus P_\infty)} + \right. \\
				& \left.
				\qquad\qquad\qquad				+ \| \Omega^{-2} (\Delta_{g^{GH}_\epsilon} - \Delta_{g_\epsilon} )\|_{op} \cdot \| u\|_{L^2_{\delta}(P''_\infty)} 
				\right].
			\end{align*}
			For any model end we consider, one has $\lim_{\rho \to \infty} |\nabla^k_{cf} (g_\epsilon - g^{GH}_\epsilon)| = 0$ for all $k \in \N$ and this limit is uniform in the collapsing parameter $\epsilon$.
			Therefore, the operator norm of $\| \Omega^{-2} (\Delta_{g_\epsilon} - \Delta_{g^{GH}_\epsilon})\|_{cf}$ can be chosen arbitrary small by translating the domains of $P_\infty$, $P'_\infty$ and $P''_\infty$. Hence, by reordering, enlarging $C'$ and redefining $P''_\infty$ as $P'_\infty$,
			\begin{align*}
				\| u\|_{W^{2,2}_{\delta}(P_\infty)} 
				\le& C' \left[
				\| \Omega^{-2} \Delta_{g_\epsilon} u\|_{L^{2}_{\delta}(P'_\infty)}
				+ \| u\|_{L^{2}_{\delta}(P'_\infty \setminus P_\infty)}
				\right],
			\end{align*}			
			which is one of the estimations we need to show. The other estimates follow by a similar argument.
		\end{proof}			
		\noindent
		Identical to the proofs of Theorems 9.1.1 and 9.2.1 in \cite{Pacard2008}, we now conclude
		\begin{corollary} 
			\label{cor:fredholmness}
			Let $W^{2, 2}_{\delta, 0}(P'_\infty)$ be the space of all $W^{2,2}_{\delta}(P'_\infty)$ functions that satisfy $u|_{\del P'_\infty} = 0$.
			Under the conditions described in Theorem \ref{thm:analysis-asymptotic:special-perp-schauder-polynomial-case}, the operator
			$$
			\Omega^{-2} \Delta \colon W^{2,2}_{\delta, 0}(P'_\infty) \to L^{2}_{\delta}(P'_\infty)
			$$
			is Fredholm.
		\end{corollary}
	\section{Invertibility of the Laplacian}
		\label{subset:analysis-asymptotic:fredholm-theory}
		In order to understand the index of $\Omega^{-2} \Delta_{g_\epsilon}$, we will study the kernel and co-kernel of $\Omega^{-2} \Delta_{g^{GH}_\epsilon}$ in more detail.
		In the case $\delta < 0$, injectivity follows from the maximum principle, because for this case functions inside $C^{k,\alpha}_\delta$ must decay. In the next proposition get an improved version for $S^1$ non-invariant functions. 
		\begin{proposition}
			\label{prop:analysis-asymptotic:laplace-equation-dirichlet-perp-holder}
			One can define $P'_\infty$ such that there exists a $\tilde \delta > 0$, such that for any $\delta < \tilde \delta$ and $\alpha \in (0,1)$ there are no non-zero $u \in C^{2,\alpha}_\delta (P'_\infty)$ that satisfy
			\begin{align*}
				u \text{ is } S^1 \text{ non-invariant}, \qquad
				\Delta_{g^{GH}_\epsilon} u = 0, \quad\text{ and} \quad
				u|_{\del P'} = 0.
			\end{align*}
		\end{proposition}
		\begin{proof}
			Recall we defined $P'_\infty$ to be the circle bundle over $[R_1, \infty) \times \Sigma$, where $\Sigma$ is a compact set. Let $R > R_1$ and define $U_r \subset P'_\infty$ to be the circle bundle over $[R_1, r] \times \Sigma$.
			Using integration by parts, one can show that for any harmonic function $u$ on $U_r$ and $\delta \in \R$,
			\begin{align*}
				\| \d (e^{- 2\delta \rho} u) \|^2_{L^2_{GH}(U_r)} 
				=& \int_{\del U_r} e^{- 4\delta \rho} u\:  *^{GH} \d  u
				+ 4\delta^2 \cdot \|  e^{- 2\delta \rho} u \: \d \rho   \|^2_{L^2_{GH}(U_r)}.
			\end{align*}
			With respect to $g^{GH}_\epsilon$, the norm of $\d \rho$ is $\frac{1}{\sqrt h_\epsilon}$ or $\frac{1}{r \sqrt h_\epsilon}$ when $B = \R\times T^2$ or $B \not = \R \times T^2$ respectively. In any case this is bounded by one.
			By the Poincaré inequality, $\| e^{- 2\delta \rho} u \|^2_{L^2_{GH}(U_r)}  \le 2 \pi \epsilon \cdot \| \d (e^{- 2\delta \rho} u) \|^2_{L^2_{GH}(U_r)}$ and hence
			\begin{align*}
				(1 - 8 \pi \epsilon \delta^2\cdot \|  \d \rho \|^2_{C^0_{GH}(P'_\infty)} ) \cdot \| \d(e^{- 2\delta \rho} u) \|^2_{L^2_{GH}(U_r)}  \le  \int_{\del U_r} e^{- 4 \delta \rho} u\:  *^{GH} \d  u.
			\end{align*}
			We pick $R_1$ and $\delta$ such that $8 \pi \epsilon \delta^2 \cdot \|  \d \rho \|^2_{C^0_{GH}(P'_\infty)}  < 1$.
			\\
			
			\noindent
			Finally, we use the fact that $u$ vanishes on $\del P'$. We  are left with
			\begin{align*}
				\int_{\del U_r} e^{- 4 \delta \rho} u\:  *^{GH} \d  u =& \begin{cases}
					\epsilon
					e^{(1- 4 \delta) R} u(R) \frac{\del u}{\del \rho}(R)\:  
					\int_{\Sigma}  
					\Vol_{S^2} \wedge  \d t &\text{if } B = \R^3 \\
					\epsilon e^{- 4 \delta R} u(R) \frac{\del u}{\del \rho}(R) \int_{\Sigma}\:
					\Vol_{S^1 \times S^1} \wedge  \d t &\text{if } B = \R^2 \times S^1 \\
					\epsilon \: e^{- 4 \delta R} \:  u(R) \frac{\del u}{\del \rho}(R)\:  
					\int_{\Sigma}\Vol_{T^2} \wedge \d t &\text{otherwise}.
				\end{cases}
			\end{align*}
			When $u \in C^{k, \alpha}_\delta(P'_\infty)$, there is a constant $C > 0$ such that
			\begin{align*}
				\int_{\del U_r} e^{- 4 \delta \rho} u\:  *^{GH} \d  u \le&  C \cdot \begin{cases}
					e^{(1- 2 \delta) R}  &\text{if } B = \R^3 \\
					e^{- 2 \delta R} & \text{otherwise}.
				\end{cases}
			\end{align*}
			This vanishes at infinity when $\delta > 0$ or when $\delta > \frac{1}{2}$. This implies that in the limit $r \to \infty$, $\| \d(e^{- 2\delta \rho} u )\|^2_{L^2_{GH}(U_r)}  = 0$ and hence $u$ must be a multiple of $e^{\delta \rho}$. The only $S^1$ non-invariant function that satisfies this is the constant zero function.
		\end{proof}

		\noindent
		Using Theorem \ref{thm:analysis-asymptotic:schauder-global-holder} and Theorem \ref{thm:analysis-asymptotic:schauder-global-nash-moser}, we can extend this result to Sobolev spaces:
		\begin{corollary}
			\label{cor:analysis-asymptotic:laplace-equation-dirichlet-perp-Schauder}
			One can define $P'_\infty$ such that there exists a $\tilde \delta > 0$, such that for any $\delta < \tilde \delta$ and $\alpha \in (0,1)$ there are no non-zero $u \in W^{2,2}_\delta (P'_\infty)$ that satisfy
			\begin{align*}
				u \text{ is } S^1 \text{ non-invariant}, \qquad
				\Delta_{g^{GH}_\epsilon} u = 0, \quad\text{ and} \quad
				u|_{\del P'} = 0.
			\end{align*}
		\end{corollary}
			\noindent
		Proposition \ref{prop:analysis-asymptotic:laplace-equation-dirichlet-perp-holder} is not true for $S^1$ invariant functions. However, for these functions the Laplace equation can be explicitly be solved using the Fourier decomposition. For $\delta < 1$, they are the following:
		\begin{proposition}
			\label{prop:analysis-asymptotic:laplace-equation-dirichlet}
			Any $u \in C^{2, \alpha}_\delta (P'_\infty)$ or $u \in W^{2, 2}_\delta (P'_\infty)$ that satisfies
			\begin{align*}
				u \text{ is } S^1 \text{invariant}, \qquad
				\Delta_{g^{GH}_\epsilon} u = 0, \quad\text{ and} \quad
				u|_{\del P'} = 0.
			\end{align*}
			will vanish when $\delta < 0$.
			For $\delta \in (0, 1)$, $u$ must be of the form
			$$
			u =\begin{cases}
				\lambda + \mu \cdot e^{- \rho} &\text{if } B = \R^3, \\
				\lambda + \mu \cdot \rho &\text{otherwise},
			\end{cases}
			$$ where $\lambda, \mu \in \R$ are chosen such that $ u|_{\del P'_\infty} = 0$.
		\end{proposition}				
		\noindent
		When one uses Sobolev spaces, one can calculate the cokernel of an operator by studying the kernel of its formal adjoint. In the next proposition we make this precise. Combining this with our knowledge of the kernel from Proposition 	\ref{prop:analysis-asymptotic:laplace-equation-dirichlet} we will get an explicit description of the range.
		
		\begin{proposition}
			\label{prop:analysis-asymptotic:cokernel}
			Let $W^{2, 2}_{\delta, 0}(P'_\infty)$ be the space of all $W^{2,2}_{\delta}(P'_\infty)$ functions that satisfy $u|_{\del P'_\infty} = 0$. The formal adjoint of $L_\delta$ is $L_{-(\delta + 1)}$ when $B = \R^3$ and $L_{- \delta}$ else.
			Hence, under the conditions described in Theorem \ref{thm:analysis-asymptotic:special-perp-schauder-polynomial-case}, $f \in L^2_\delta(P'_\infty)$ lies in the image of
			$$
			\Omega^{-2} \Delta_{g_\epsilon} \colon W^{2,2}_{\delta, 0}(P'_\infty) \to L^{2}_{\delta}(P'_\infty)
			$$
			if and only if $\langle f, e^{\rho}\cdot v \rangle_{\tilde L^2(P'_\infty)} = 0$ (or $\langle f, v \rangle_{\tilde L^2(P'_\infty)} = 0$ when $B \not = \R^3$) for all 
			$$
			v \in \ker \Omega^{-2} \Delta_{g_\epsilon}\colon\begin{cases}
				W^{2,2}_{-(\delta + 1), 0}(P'_\infty) \to L^{2}_{-(\delta + 1)}(P'_\infty) &\text{if } B = \R^3 \\
				W^{2,2}_{-\delta, 0}(P'_\infty) \to L^{2}_{-\delta}(P'_\infty) &\text{otherwise}.
			\end{cases}
			$$
		\end{proposition}

		\begin{proof}
		The formal adjoint can be explicitly be calculated by considering $\langle L_\delta u, v \rangle_{\tilde L^2(P'_\infty)} = \langle  u,L_\delta^* v \rangle_{\tilde L^2(P'_\infty)}$ for any pair of compactly supported smooth functions $u$ and $v$ on $P'_\infty$.
	\end{proof}

		\noindent
		When $B = \R^3$ the operator is injective for $\delta < 0$. According to Proposition \ref{prop:analysis-asymptotic:cokernel}, it must be surjective when $\delta > -1$. Hence it is an isomorphism for $\delta \in (-1,0)$. However, when $B \not = \R^3$ there is no $\delta \in \R$ such that $\Omega^{-2} \Delta_{g^{GH}_\epsilon}$ is injective and surjective at the same time. Hence we need to manually enlarge the domain without adding new elements to the kernel. We claim that when $B \not = \R^3$, $\delta  <0$ and $|\delta| \ll 1$ the operator 
		$$
		\Omega^{-2} \Delta_{g^{GH}_\epsilon} \colon W^{2,2}_\delta(P'_\infty) \oplus \R \rho \to L^2_\delta(P'_\infty)
		$$
		with Dirichlet boundary conditions is the isomorphism we are looking for. 
		\begin{proposition}
			\label{prop:analysis-asymptotic:laplace-bijectivity-dirichlet}
			Let $\delta \in (-1, 0)$ with $|\delta|$ sufficiently small.
			For any $f \in L^2_{\delta}(P'_\infty)$ there exists a unique $u \in W^{2, 2}_{\delta}(P'_\infty)$ or $u \in W^{2, 2}_{\delta}(P'_\infty) \oplus \R \rho$ such that 
			\begin{align*}
				\Omega^{-2} \Delta_{g^{GH}_\epsilon} u = f \quad \text{and} \quad
				u|_{\del P'_\infty} = 0 
			\end{align*}
			when $B = \R^3$ or $B \not= \R^3$ respectively.
		\end{proposition}
		\begin{proof}
			We only prove the case $B \not = \R^3$. Let $u + \lambda \rho \in W^{2, 2}_{\delta}(P'_\infty) \oplus \R \rho$ such that $\Delta (u + \lambda \rho) = 0$ and $u + \lambda \rho|_{\del P'} = 0$. Because $W^{2, 2}_{\delta}(P'_\infty) \oplus \R \rho \subset W^{2, 2}_{-\delta}(P'_\infty)$, there exist $\alpha, \beta \in \R$ such that 
			\begin{equation*}
				u + \lambda \rho = \alpha + \beta \rho,
			\end{equation*}
			Comparing decay rates, we conclude $\lambda = \beta$ and $u = \alpha$.
			The only constant function that is part of $W^{2, 2}_\delta(P'_{\infty})$ is the constant zero function and therefore $\alpha = 0$. The boundary condition forces $\beta = 0$. This proves the injectivity of $\Omega^{-2} \Delta_{g^{GH}_\epsilon}$.
			\\
			
			\noindent
			To show surjectivity we first set up some notation: Let $\alpha \in \R$ be such that $\alpha + \rho$ vanishes on the boundary of $P'_\infty$. Let $\chi$ be a smooth bump function on $P'_\infty$ such that $\chi|_{\del P'_\infty} = 1$ and assume that $\langle \Omega^{-2}\Delta_{g^{GH}_\epsilon}(\chi \rho) , \alpha + \rho \rangle \not= 0$. 
			\\
			
			\noindent
			Let $f \in L^2_{\delta}(P'_\infty)$ and choose $\beta \in \R$ such that $\langle f + \beta \cdot \Omega^{-2}\Delta_{g^{GH}_\epsilon}(\chi \rho), \alpha + \rho \rangle = 0$. By Proposition \ref{prop:analysis-asymptotic:cokernel}, there exists some $\hat u \in W^{2,2}_{\delta,0}(P'_\infty)$ such that 
			$$
			\Omega^{-2} \Delta_{g^{GH}_\epsilon} (\hat u) = f + \beta \cdot \Omega^{-2}\Delta_{g^{GH}_\epsilon}(\chi \rho),
			$$
			because $a + \rho$ spans the kernel of $\Omega^{-2} \Delta_{g^{GH}_\epsilon} \colon W^{2,2}_{-\delta}(P'_\infty) \to L^2_{- \delta}(P'_\infty)$. By construction $u := \hat u - \beta \cdot \chi \rho$ is an element of $W^{2,2}_\delta(P'_\infty)$ and solves $\Omega^{-2}\Delta_{g^{GH}_\epsilon} u = f$.
			To satisfy the boundary condition we consider $u + \beta \cdot \rho$ instead. Because $\rho$ is harmonic, $u + \beta \cdot \rho$ is still a solution for $f$ and,
			$$
			u + \beta \rho |_{\del P'_{\infty}} = \hat{u} + \beta(1 - \chi) \rho |_{\del P'_{\infty}} = 0.
			$$
			This proves surjectivity.
		\end{proof}
		\begin{corollary}
		\label{cor:analysis-asymptotic:laplace-bijectivity-dirichlet}
		Let $\delta \in (-1, 0)$ with $|\delta|$ sufficiently small.
		One can define $P'_\infty$ such that for any $f \in L^2_{\delta}(P'_\infty)$ there exists a unique $u \in W^{2, 2}_{\delta}(P'_\infty)$ or $u \in W^{2, 2}_{\delta}(P'_\infty) \oplus \R \rho$ such that 
		\begin{align*}
			\Omega^{-2} \Delta_{g_\epsilon} u = f \quad \text{and} \quad
			u|_{\del P'_\infty} = 0 
		\end{align*}
		when $B = \R^3$ or $B \not= \R^3$ respectively.
	\end{corollary}
	\begin{proof}
		This follows directly from Proposition \ref{prop:analysis-asymptotic:laplace-bijectivity-dirichlet}, because invertibility is an open condition and the operator norm of $ \Omega^{-2} (\Delta_{g_\epsilon} - \Delta_{g^{GH}_\epsilon})$ can be made arbitrary small by changing the domain of $P'_\infty$. 
	\end{proof}
\subsection{Invertibility in Sobolev spaces}
\label{sec:Fredholm-global}
Having determined the (co)-kernel of the Laplacian on $P_\infty$, we finally are able to study the Laplacian on $M_\epsilon$.  Because $M_\epsilon$ is the union of the model end $P_\infty$ and some compact set, elliptic regularity and Fredholm results can be extended without any proof. We only need to study the kernel and range of $\Omega^{-2} \Delta_{g_\epsilon}$.
\\

\noindent
When $\delta < 0$, functions inside $W^{2,2}_\delta(M_\epsilon)$ are forced to decay at infinity. Hence when $\delta < 0$, the kernel of $\Omega^{-2}\Delta_g$ is zero due to the maximum principle. In Proposition \ref{prop:analysis-asymptotic:cokernel} it is shown that the formal adjoint of $L_\delta$ near infinity is $L_{-(1+\delta)}$ when $B = \R^3$ or $L_{- \delta}$ else. From this we can directly conclude for which $\delta$ the co-kernel is empty:
\begin{lemma}
	\label{lem:global-analysis:injectivity-delta-below-zero}
	When $B = \R^3$, the operator $\Omega^{-2} \Delta_{g_\epsilon} \colon W^{2,2}_\delta(M_{\epsilon}) \to L^2_\delta(M_{\epsilon})$ is an isomorphism for $\delta \in (-1,0)$.
	When $B \not = \R^3$, the operator $\Omega^{-2} \Delta_{g_\epsilon} \colon W^{2,2}_\delta(M_{B,n}) \to L^2_\delta(M_{B,n})$ is injective when $\delta <0$ and surjective when $\delta > 0$.
\end{lemma}
\noindent We can't improve Lemma \ref{lem:global-analysis:injectivity-delta-below-zero}, because there will be an index jump at $\delta = 0$ due to the constant functions. In summary, according to Lemma \ref{lem:global-analysis:injectivity-delta-below-zero}, there always exists an inverse, but this inverse might have the wrong decay rate. As shown in Proposition \ref{prop:analysis-asymptotic:laplace-bijectivity-dirichlet}, we can remedy this by adding a certain smooth function $\phi$ to our domain. We require $\phi$ to be $\rho$ near infinity and we want that $\phi$ vanishes inside.
\begin{proposition}
	\label{prop:global-analysis:laplace-surjectivive-not-R3}
	Assume that $B \not= \R^3$ and assume that $g_\epsilon = g^{GH}_\epsilon$ on $P'_\infty$.
	Let $\delta <0 $ with $|\delta|$ sufficiently small and $k \in \N_{\ge 2}$.
	For any $f \in W^{k-2, 2}_{\delta}(M_{\epsilon})$ there exists a $u \in W^{k, 2}_{\delta}(M_{\epsilon})\oplus \R \phi$ such that 
	\begin{align*}
		\Omega^{-2} \Delta_{g_\epsilon} u =& f.
	\end{align*}
\end{proposition}
\begin{proof}
	By elliptic regularity it is sufficient to show that $\Omega^{-2} \Delta_{g_\epsilon} \colon W^{2,2}_\delta(M_{\epsilon}) \oplus \R \phi \to L^2_\delta(M_{\epsilon})$ is surjective. Let $f \in L^2_\delta(M_{\epsilon})$. By Lemma \ref{lem:global-analysis:injectivity-delta-below-zero} there exists a $u \in W^{2,2}_{- \delta}(M_{\epsilon})$ such that 
	$$
	\Omega^{-2} \Delta u = f.
	$$
	Our goal is to show that $u \in W^{2,2}_\delta(M_{\epsilon}) \oplus \R \phi$.
	Let $\chi$ be a small bump function on $M_{\epsilon}$ that equals 1 on $\del P'_\infty$. By Corollary \ref{cor:analysis-asymptotic:laplace-bijectivity-dirichlet} there exist a function $u_\infty \in W^{2,2}_\delta(P'_\infty)$ and $\lambda \in \R$ such that
	\begin{align*}
		\Omega^{-2} \Delta_{g^{GH}_\epsilon}(u_\infty + \lambda \phi) =& f - \Omega^{-2} \Delta_{g^{GH}_\epsilon}(\chi u), \\
		(u_\infty + \lambda\phi)|_{\del P'_\infty} =& 0.
	\end{align*}
	The term $\Omega^{-2} \Delta_{g^{GH}_\epsilon}(\chi u) $ is added, because it induces the conditions
	\begin{align*}
		\Omega^{-2} \Delta_{g^{GH}_\epsilon}(u_\infty + \chi u + \lambda \phi) =& f, \\
		(u_\infty + \chi u + \lambda\phi)|_{\del P'_\infty} =& u.
	\end{align*}
	At the same time, the restriction of $u$ to the region $P'_\infty$ also satisfies
	\begin{align*}
		\Omega^{-2} \Delta_{g^{GH}_\epsilon}(u) =& f \\
		u |_{\del P'_\infty} =& u,
	\end{align*}
	and hence $u_\infty + \chi u + \lambda \phi - u$ is a harmonic function on $P'_\infty$ with Dirichlet boundary conditions. Because $W^{2,2}_\delta \oplus \R \phi$ is a subset of $W^{2,2}_{- \delta}$, and the harmonics of $W^{2,2}_{- \delta}(P'_\infty)$ are known by Proposition \ref{prop:analysis-asymptotic:laplace-equation-dirichlet},
	$$
	u_\infty + \chi u + \lambda \phi - u = \alpha + \beta \rho
	$$ 
	for some $\alpha, \beta \in \R$.
	From this we make two observations: First, $u + \alpha + (\beta - \lambda) \phi$ is an element of $W^{2,2}_{-\delta}(M_{\epsilon})$, and secondly, it is also equal to $u_\infty + \chi u \in W^{2,2}_{\delta}(P'_\infty)$. Because $M_{\epsilon}$ is the union of a compact set with $P'_\infty$ and all weighted $W^{2,2}$ norms on compact sets are equivalent,
	$$
	u + \alpha + (\beta - \lambda) \phi \in W^{2,2}_{\delta}(M_{\epsilon}).
	$$
	We conclude $u + \alpha \in W^{2,2}_\delta(M_{\epsilon}) \oplus \R \phi$ and $\Omega^{-2} \Delta_{g_\epsilon}(u + \alpha) = f$, which proves surjectivity.
\end{proof}
\begin{proposition}
	\label{prop:global-analysis:laplace-injective-not-R3}
	Assume that $B \not= \R^3$ and assume that $g_\epsilon = g^{GH}_\epsilon$ on $P'_\infty$.
	Let $\delta \in (-1, 0)$ with $|\delta|$ sufficiently small and $k \in \N_{\ge 2}$. The operator 
	$$\Omega^{-2} \Delta_{g_\epsilon} \colon W^{k, 2}_{\delta}(M_{\epsilon}) \oplus \R \phi \to L^2_\delta(M_{\epsilon})$$
	has a trivial kernel.
\end{proposition}
\begin{proof}
	\noindent
	Assume the contrary, and let $v$ be a non-zero element of $ W^{2,2}_\delta(M_{\epsilon})$ and $\lambda \in \R$ such that $\Delta(v + \lambda \phi) = 0$. If $\lambda = 0$, Lemma \ref{lem:global-analysis:injectivity-delta-below-zero} implies $v = 0$ which contradicts our assumption. Therefore, we can rescale our harmonic function such that $\lambda = 1$. \\
	
	\noindent
	We claim that our assumption implies surjectivity of $\Omega^{-2} \Delta_{g_\epsilon} \colon W^{2,2}_\delta(M_{\epsilon}) \to L^2_\delta(M_{\epsilon})$. Indeed, let $f \in L^2_\delta(M_{\epsilon})$. By Proposition 	\ref{prop:global-analysis:laplace-surjectivive-not-R3} there must be a $u \in W^{2,2}_\delta(M_{\epsilon})$ and a $\lambda \in \R$ such that $\Omega^{-2}\Delta_{g_\epsilon}(u + \lambda \phi) = f$. By our choice of $v$, we also have
	$$
	\Omega^{-2} \Delta_{g_\epsilon}(u - \lambda v) 
	= \Omega^{-2} \Delta_{g_\epsilon}(u + \lambda \phi - \lambda (v + \phi)) = f.
	$$
	Hence $u - \lambda v \in W^{2,2}_{\delta}(M_{\epsilon})$ is an inverse of $f$.
	\\
	
	\noindent
	We claim that surjectivity of $\Omega^{-2} \Delta_{g_\epsilon} \colon W^{2,2}_\delta(M_{\epsilon}) \to L^2_\delta(M_{\epsilon})$ leads to a contradiction. Indeed, when $\Omega^{-2} \Delta_{g_\epsilon}$ is surjective, then $L_\delta$ is surjective and its formal adjoint must be injective. 
	On the asymptotic part of $M_{\epsilon}$, the formal adjoint is $L_{-\delta}$ . Because $\delta < 0$, the constants are part of the kernel of $L_\delta^*$, but we just have shown that the kernel of $L_\delta^*$ is trivial. Therefore, $v$ does not exist.
\end{proof}

\noindent
Lemma \ref{lem:global-analysis:injectivity-delta-below-zero} and Propositions \ref{prop:global-analysis:laplace-surjectivive-not-R3} and  \ref{prop:global-analysis:laplace-injective-not-R3} prove Theorem \ref{thm:analysis-asymptotic:laplace-bijectivity-global} for the Sobolev norm under the extra assumption that $g_\epsilon = g^{GH}_\epsilon$ on the asymptotic region $P'_\infty$. Because invertibility is an open condition and the operator norm of $ \Omega^{-2} (\Delta_{g_\epsilon} - \Delta_{g^{GH}_\epsilon})$ can be made arbitrary small, this extra condition is superfluous.

\subsection{Invertibility in H\"older spaces}
For any $\tilde{\delta} > \delta$, the space $
C^{0}_\delta(M_{B,n})$ embeds into $L^2_{\tilde \delta}(M_{\epsilon})$. Therefore, the isomorphism from Theorem \ref{thm:analysis-asymptotic:laplace-bijectivity-global} for the Sobolev case and the regularity result from Theorem \ref{thm:analysis-asymptotic:schauder-global-nash-moser} imply that for any $f \in C^{0, \alpha}_\delta(M_\epsilon)$ and $\tilde \delta > \delta$, there exists a $u \in C^{2, \alpha}_{\tilde \delta}(M_\epsilon)$  (or $u \in C^{2, \alpha}_{\tilde \delta}(M_\epsilon) \oplus \R \phi$) such that $\Omega^{-2} \Delta_{g_\epsilon} u = f$. To make $\Omega^{-2} \Delta_{g_\epsilon}$ into an actual isomorphism, we need to regain the weight that we have lost in the embedding, i.e. we need to show that $u \in C^{2, \alpha}_{\delta}(M_\epsilon)$  (or $u \in C^{2, \alpha}_{\delta}(M_\epsilon) \oplus \R \phi$). 
\\

\noindent
In order to regain the weight, we are going to study the family of functions $u_{\tilde \delta}$ that solve $\Omega^{-2} \Delta_{g_\epsilon} u_{\tilde \delta} = f$ for some fixed $f \in C^{0, \alpha}_\delta(M_\epsilon)$ and we will see if we can take the limit $\tilde \delta$ to $\delta$. At first sight this limit should not converge in $C^{2, \alpha}_{\tilde \delta}(M_\epsilon)$ or $C^{2, \alpha}_{\tilde \delta}(M_\epsilon) \oplus \R \phi$. Namely, the weight function $e^{\tilde \delta \rho}$ does not converge to $e^{\delta \rho}$ \textit{uniformly}. Therefore, any $C^{2, \alpha}_\delta$ estimate of $u_{\tilde \delta}$ will likely diverge. We will circumvent this issue by considering pointwise convergence and only use uniform convergence on compact sets. For this to work we first need to extend Theorem \ref{thm:analysis-asymptotic:special-perp-schauder-polynomial-case} globally and study the behavior of its constants under perturbation of weight.

\begin{lemma}
	\label{lem:uniformity-delta}
	Let $\delta_{\min},\delta_{\max} \in \R$ such that $[\delta_{\min}, \delta_{\max}] \subset \R \setminus \Z$.
	On top of the conditions of Theorem \ref{thm:analysis-asymptotic:special-perp-schauder-polynomial-case}, assume that $ \delta \in [\delta_{\min}, \delta_{\max}]$.
	There exists a compact set $K$ and a constant $C > 0$, depending on $\epsilon$, $\delta_{\min}$ and $\delta_{\max}$ such that for any $u \in C^{2, \alpha}_{\delta}(M_\epsilon)$,
	\begin{align*}
		\| u\|_{C^{2,\alpha}_{\delta}(M_\epsilon)} \le& C \left[
		\| \Omega^{-2}\Delta_{g_\epsilon} u\|_{C^{0, \alpha}_{\delta}(M_\epsilon)}
		+ \| u\|_{C^0_{\delta}(K)} 
		\right].
	\end{align*}
\end{lemma}
\begin{remark}
	Before we were able to to show that all are estimates are uniform in $\epsilon$. This was possible, because we controlled the geometry of the asymptotic region of $M_\epsilon$. If we have a uniform elliptic regularity estimate on the interior of $M_\epsilon$ -- for example, we have uniform bounded geometry on the interior of $M_\epsilon$ -- then the constant in the above lemma can be chosen uniform w.r.t. $\epsilon$. As we haven't put any conditions on the compact region of $M_\epsilon$, this uniformity of the estimate cannot be guaranteed.
\end{remark}
\begin{proof}[Proof of Lemma \ref{lem:uniformity-delta}]
	Let $P_\infty \subset P'_\infty$ be the asymptotic regions defined in Theorem \ref{thm:analysis-asymptotic:special-perp-schauder-polynomial-case}.
	Let $K' \subset K''$ compact subsets of $M_\epsilon$, such that they both cover $M_\epsilon \setminus P_\infty$. By elliptic regularity of $\Omega^{-2} \Delta_{g_\epsilon}$, there exists a constant $C> 0$, depending on $\epsilon$ and $\delta$, such that 
	$$
	\| u\|_{C^{2,\alpha}_{\delta}(K')} \le C \left[
	\| \Omega^{-2}\Delta_{g_\epsilon} u\|_{C^{0, \alpha}_{\delta}(K'')}
	+ \| u\|_{C^0_{\delta}(K'')} 
	\right].	
	$$
	Combining this with the estimates of Theorem \ref{thm:analysis-asymptotic:special-perp-schauder-polynomial-case}, we get
	\begin{align*}
			\| u\|_{C^{2,\alpha}_{\delta}(K' \cup P_\infty)} \le C \left[
		\| \Omega^{-2}\Delta_{g_\epsilon} u\|_{C^{0, \alpha}_{\delta}(K'' \cup P'_\infty)}
		+ \| u\|_{C^0_{\delta}(K'' \cup (P'_\infty \setminus P_\infty))} 
		\right].	
	\end{align*}
	Because $P'_\infty \setminus P_\infty$ can be chosen compact, we set $K = K'' \cup (P'_\infty \setminus P_\infty)$ and conclude
	\begin{align}
		\label{eq:special-fredholm-need-to-check-uniformity-delta}
		\| u\|_{C^{2,\alpha}_{\delta}(M_\epsilon)} \le C \left[
		\| \Omega^{-2}\Delta_{g_\epsilon} u\|_{C^{0, \alpha}_{\delta}(M_\epsilon)}
		+ \| u\|_{C^0_{\delta}(K))} 
		\right].	
	\end{align}
	We only need to show $C$ can be chosen uniformly in $\delta$. We prove this by using that $L_\delta$ is uniform in $\delta$. Indeed, assume that $C$ is not uniform in $\delta$.
	Then, there must be sequences $u_i \in C^{2, \alpha}_{cf}(M_{\epsilon})$ and $\delta_i \in \R$ such that
	\begin{align*}
		\|u_i\|_{C^{2, \alpha}_{cf}(M_{\epsilon})} &= 1,  &
		\|L_{\delta_i} u_i\|_{C^{0, \alpha}_{cf}(M_{\epsilon})} &\to 0, \\
		\|u_i\|_{C^{0}_{cf}(M_{\epsilon})} &\to 0,  &
		\delta_i &\to \delta_{\lim} \in [\delta_{\min}, \delta_{\max}].
	\end{align*}
	We apply Equation \ref{eq:special-fredholm-need-to-check-uniformity-delta} on $u_i$ with the limiting weight $\delta_{\lim}$, which yields
	\begin{align}
		\| u_i\|_{C^{2,\alpha}_{cf}(M_{\epsilon})} 
		&\le C(\epsilon, \delta_{\lim}) \left[
		\| L_{\delta_{\lim}} u_i\|_{C^{0,\alpha}_{cf}(M_{\epsilon})}
		+ \| u_i\|_{C^{0}_{cf}(M_{\epsilon})}
		\right] \notag\\
		&\le C(\epsilon, \delta_{\lim}) \left[
		\| L_{\delta_i} u_i\|_{C^{0,\alpha}_{cf}(M_{\epsilon})}
		+ \| (L_{\delta_{\lim} }  -L_{\delta_i}) u_i\|_{C^{0,\alpha}_{cf}(M_{\epsilon})}
		+ \| u_i\|_{C^{0}_{cf}(M_{\epsilon})}
		\right]. \label{eq:estimate-uniformity-delta-partly}
	\end{align}
	For any $\delta \in \R$, 
	$$
	L_\delta u_i 
	= e^{- \delta \rho} \Omega^{-2} \Delta_{g_\epsilon}(e^{\delta \rho} u)
	= \Omega^{-2} \Delta_{g_\epsilon}(u) 
	+  \delta \: u \cdot  \Omega^{-2} \Delta_{g_\epsilon} \rho
	- \delta^2  u \cdot  \| \d \rho \|^2_{g_{cf}}
	 - 2 \delta  \langle \d u, \d \rho \rangle_{g_{cf}}
	$$
	which imply that there exists a uniform constant $C' > 0$ such that 
	$$
	\| (L_{\delta_{\lim} }  -L_{\delta_i}) u_i\|_{C^{0,\alpha}_{cf}(M_{\epsilon})} \le C \:|\delta_{\lim} - \delta_{i}| \cdot \| u_i \|_{C^{1,\alpha}_{cf}(M_{\epsilon})}.
	$$
	Hence, the right hand side of Equation \ref{eq:estimate-uniformity-delta-partly} converge to zero. This yields a contradiction as $	\|u_i\|_{C^{2, \alpha}_{cf}(M_{\epsilon})} = 1$ for all $i \in \N$.
\end{proof}

\noindent
With the uniform control of $C_\epsilon$ we can finally prove the bijectivity of $\Omega^{-2} \Delta_{g_\epsilon}$.
\begin{proof}[Proof of Theorem \ref{thm:analysis-asymptotic:laplace-bijectivity-global}, H\"older case]
	Let $f \in C^{k, \alpha}_\delta(M_\epsilon)$. For any $\tilde \delta \in (\delta, \delta/2)$, $f$ is an element of $L^2_{\tilde \delta}(M_{B,n})$, and so by Theorem 	\ref{thm:analysis-asymptotic:laplace-bijectivity-global} there exists a $u \in W^{2,2}_{\tilde \delta}(M_{B,n})$ such that $\Omega^{-2} \Delta_{g_\epsilon} u = f$.  By Proposition \ref{prop:analysis-asymptotic:schauder-local-nash-moser}, $u$ is an element of $C^{2, \alpha}_{\tilde \delta}(M_{B,n})$ and because $\Omega^{-2} \Delta_{g_\epsilon}$ is injective, the function $u$ does not depend on the choice of $\tilde \delta$.
	\\

	\noindent
	Using elliptic regularity it is sufficient to show $u \in C^{0}_\delta(M_\epsilon)$. So let $x \in M_\epsilon$ and consider $|e^{- \delta \rho} u| (x)$.
	We can estimate this as
	$$
	|e^{- \delta \rho} u|(x) \le |e^{(\tilde \delta - \delta) \rho}|(x) \cdot \|e^{- \tilde \delta} u\|_{C^{2, \alpha}_{cf}(M_{\epsilon})}.
	$$
	Because $\tilde \delta$ is not an indicial root, we can apply Lemma \ref{lem:uniformity-delta}:
	\begin{align*}
		|e^{- \delta \rho} u|(x) &\le |e^{(\tilde \delta - \delta) \rho}|(x) \cdot 
		C \left[
		\|e^{- \tilde \delta \rho} \: f\|_{C^{0, \alpha}_{cf}(M_{\epsilon})} + 
		\|e^{- \tilde \delta \rho} u\|_{C^{0}(K)}
		\right] \\
		&\le |e^{(\tilde \delta - \delta) \rho}|(x) \cdot 
		C \left[
		\|e^{(\delta - \tilde \delta) \rho} \|_{C^{0, \alpha}_{cf}(M_{\epsilon})} \cdot 
		\|f\|_{C^{2, \alpha}_{\delta}(M_{\epsilon})} + 
		\|e^{- \tilde \delta\rho }\|_{C^{0}(K)}
		\cdot
		\|u\|_{C^{0}(K)}
		\right].
	\end{align*}
	The terms $\|e^{- \tilde \delta\rho }\|_{C^{0}(K)}$, $\|e^{(\delta - \tilde \delta) \rho} \|_{C^{0, \alpha}_{cf}(M_{\epsilon})}$ and $C$ are all uniformly bounded w.r.t. $\tilde \delta$: For the first term can be estimated explicitly, for the second term this follows due to the fact that $e^{(\delta - \tilde \delta) \rho} $ decays when $\tilde \delta > \delta$ and for the last term this is shown in Lemma \ref{lem:uniformity-delta}. Therefore, there exists a constant $C'$ that depends on $C$, $\|f\|_{C^{2, \alpha}_{\delta}(M_{B,n})}$ and $\|u\|_{C^{0}(K)}$ such that $|e^{- \delta \rho} u|(x)
	\le C' \cdot |e^{(\tilde \delta - \delta) \rho}|(x)$.
	\\

	\noindent
	For each $x \in M_{\epsilon}$, we pick $\tilde \delta > \delta$ such that $|e^{(\tilde \delta - \delta) \rho}|(x) \le 2$. This gives us an estimate of $|e^{- \delta \rho} u|(x)$ which does not depend on $\tilde \delta$. Therefore, $\|u\|_{C^0_\delta(M_{\epsilon})} = \sup_{x \in M_{\epsilon}} |e^{- \delta \rho} u|(x)
	\le  2 C' < \infty$.
\end{proof}

\newpage
\bibliography{Library.bib}

\end{document}